\documentclass[11pt,reqno]{amsart}
\usepackage{amsmath,amssymb,amsthm,pinlabel}
\usepackage[final, 
colorlinks=true,
pdftitle={Atoroidal subgroup alternative for Out(F_N)},
pdfauthor={Matt Clay, Caglar Uyanik}]{hyperref}
\usepackage[full]{textcomp}
\usepackage[osf]{newpxtext} 
\usepackage{cabin} 
\usepackage[varqu,varl]{inconsolata} 
\usepackage[bigdelims,vvarbb]{newpxmath} 
\usepackage[cal=boondoxo]{mathalfa} 
\usepackage{fullpage} 
 
\usepackage{tikz}
\usetikzlibrary{matrix,arrows,decorations.pathmorphing,patterns}

\newcounter{alphathm}

\theoremstyle{plain}
\newtheorem{theorem}{Theorem}
\newtheorem{alphathm}[alphathm]{Theorem}
\newtheorem{assumption}[theorem]{Standing assumption}
\newtheorem{lemma}[theorem]{Lemma}
\newtheorem{proposition}[theorem]{Proposition}
\newtheorem{corollary}[theorem]{Corollary}

\newtheorem*{claim*}{Claim}

\theoremstyle{definition}
\newtheorem{definition}[theorem]{Definition}

\newtheorem{remark}[theorem]{Remark}

\newtheorem{remark-convention}[theorem]{Remark-Convention}

\numberwithin{equation}{section}
\numberwithin{theorem}{section}

\newcommand{\fakeenv}{} 

\newenvironment{restate}[2]  
{ 
 \renewcommand{\fakeenv}{#2} 
 \theoremstyle{plain} 
 \newtheorem*{\fakeenv}{#1~\ref{#2}} 
 \begin{\fakeenv}
}
{
 \end{\fakeenv}
}

\newcommand{\RR}{\mathbb{R}}
\newcommand{\PP}{\mathbb{P}}
\newcommand{\ZZ}{\mathbb{Z}} 

\newcommand{\calA}{\mathcal{A}}

\newcommand{\calC}{\mathcal{C}}

\newcommand{\calF}{\mathcal{F}}

\newcommand{\calH}{\mathcal{H}}

\newcommand{\calP}{\mathcal{P}}

\newcommand{\frakg}{\mathfrak{g}}

\newcommand{\hDelta}{\widehat \Delta}
\newcommand{\hU}{\widehat U}
\newcommand{\hV}{\widehat V}

\newcommand{\abs}[1]{\left\lvert {#1} \right\rvert} 
\newcommand{\wght}[1]{\lvert {#1} \rvert} 
\newcommand{\I}[1]{\langle #1 \rangle}

\newcommand{\from}{\colon\thinspace}

\newcommand{\inv}{^{-1}}

\newcommand{\FN}{F_{N}}

\newcommand{\zF}{\mathcal{ZF}}

\newcommand{\A}{A}
\newcommand{\B}{B}

\newcommand{\param}%
	{{\mathchoice{\mkern1mu\mbox{\raise2.2pt\hbox{$\centerdot$}}\mkern1mu}%
	{\mkern1mu\mbox{\raise2.2pt\hbox{$\centerdot$}}\mkern1mu}%
	{\mkern1.5mu\centerdot\mkern1.5mu}{\mkern1.5mu\centerdot\mkern1.5mu}}}

\DeclareMathOperator{\Aut}{Aut}
\DeclareMathOperator{\IA}{IA}
\DeclareMathOperator{\Inn}{Inn}

\DeclareMathOperator{\Mod}{Mod}
\DeclareMathOperator{\Out}{Out}
\DeclareMathOperator{\GL}{GL}

\DeclareMathOperator{\Cone}{Cone}
\DeclareMathOperator{\Curr}{Curr}

\DeclareMathOperator{\EG}{EG}
\DeclareMathOperator{\NEG}{NEG}
\DeclareMathOperator{\CT}{CT}

\DeclareMathOperator{\supp}{supp}

\newcommand{\PCurr}{\PP {\rm Curr}}

\begin{document}


\title{Atoroidal dynamics of subgroups of $\Out(\FN)$}

\author[M.~Clay]{Matt Clay}
\address{Department of Mathematical Sciences \\
  University of Arkansas\\
  Fayetteville, AR 72701}
\email{\href{mailto:mattclay@uark.edu}{mattclay@uark.edu}}

\author[C.~Uyanik]{Caglar Uyanik}
\address{Department of Mathematics \\
Yale University\\
10 Hillhouse Avenue\\
New Haven, CT 06520}
\email{\href{mailto:caglar.uyanik@yale.edu}{caglar.uyanik@yale.edu}}

\thanks{\tiny M.C. is partially supported by the Simons Foundation.}

\begin{abstract} 
We show that for any subgroup $\calH$ of $\Out(\FN)$, either $\calH$ contains an atoroidal element or a finite index subgroup $\calH'$ of $\calH$ fixes a nontrivial conjugacy class in $\FN$.  This result is an analog of Ivanov's subgroup theorem for mapping class groups and Handel--Mosher's subgroup theorem for $\Out(\FN)$ in the setting of irreducible elements. 
\end{abstract}

\maketitle

\section*{Introduction}

Let $S$ be an orientable surface of finite type with $\chi(S) < 0$ and $f \from S \to S$ be an orientation preserving homeomorphism.  Nielsen--Thurston classification states that after replacing $f$ with an isotopic homeomorphism, there is an invariant collection of disjoint essential simple closed curves $C$ (possibly empty) so that the complement of an open collar neighborhood of $C$ decomposes into invariant subsurfaces (possibly disconnected), where the restriction of $f$ to each subsurface is either finite order or pseudo-Anosov~\cite{CB,Th}.  In particular, if the action of $f$ on the set of isotopy classes of essential simple closed curves does not have a finite orbit, then $f$ is isotopic to a pseudo-Anosov homeomorphism.  For our purposes, we will not need the definition of a pseudo-Anosov homeomorphism but we note that such homeomorphisms have a very rigid structure and possess desirable dynamical properties.  One such example is a theorem of Thurston that states that the $3$--manifold $M_{f}$, called the \emph{mapping torus of $f$}, obtained from $S \times [0,1]$ by gluing $S \times \{1\}$ to $S \times \{0\}$ via $f$, admits a hyperbolic structure if and only if $f$ is isotopic to a pseudo-Anosov homeomorphism~\cite{Thgeom}.  The importance of Thurston's result is magnified by the recent breakthrough results of Agol proving that every closed hyperbolic $3$--manifold has a finite cover that fibers over the circle, i.e., can be obtained by the above construction~\cite{Agol}. 

Ivanov strengthened the Nielsen--Thurston classification of homeomorphisms to subgroups of the mapping class group $\Mod(S)$, the group of isotopy classes of orientation preserving homeomorphisms of $S$.  Specifically, he proved that if the action of a subgroup $\calH < \Mod(S)$ on the set of isotopy classes of essential simple closed curves does not have a finite orbit, then $\calH$ contains a pseudo-Anosov element, i.e., the isotopy class of a pseudo-Anosov homeomorphism~\cite{Iva}.  A priori, each element in $\calH$ could have a finite orbit and yet the subgroup might not have a finite orbit.  What Ivanov proves is that if two elements in $\calH$ have sufficiently transverse finite orbits (in a precise sense), then some product of their powers is pseudo-Anosov.  Ivanov accomplishes this using classical ping-pong and other dynamical arguments on the space of projectivized measured laminations on $S$.

The outer automorphism group of a non-abelian free group $\FN$ of finite rank is the quotient $\Out(\FN) = \Aut(\FN)/\Inn(\FN)$.  This group is closely related to $\Mod(S)$, in particular by the Dehn--Nielsen--Baer theorem, see~\cite{FarbMargalit}.  During the last 30 years, the development of the theory of $\Out(\FN)$ has closely followed that of $\Mod(S)$, and to some extend that of $\GL(n,\ZZ)$ as well.  Examples of this beneficial analogy include the introduction of the Culler--Vogtmann outer space~\cite{CV}, the construction of train-track representatives~\cite{BH92} and more recently an investigation into the geometry of the free factor and free splitting complexes~\cite{BF14,ar:HM13}.

The notion of a pseudo-Anosov element in $\Mod(S)$ has two analogs in $\Out(\FN)$.  One of these uses the characterization of pseudo-Anosovs as the (infinite order) elements in $\Mod(S)$ that do not restrict to a proper subsurface.  An outer automorphism $\varphi \in \Out(\FN)$ is called \emph{fully irreducible} if no positive power of $\varphi$ fixes the conjugacy class of a proper free factor, i.e., the action of $\varphi$ on the set of conjugacy classes of proper free factors does not have a finite orbit (see Section~\ref{sec:outf} for complete definitions).  Like pseudo-Anosov elements, these outer automorphisms have a very rigid structure and possess desirable dynamical properties.  

The other analog uses the characterization of pseudo-Anosovs as the elements whose mapping torus admits a hyperbolic metric.  An outer automorphism $\varphi \in \Out(\FN)$ is called \emph{atoroidal} if no positive power of $\varphi$ fixes the conjugacy class of a nontrivial element in $\FN$, i.e., the action of $\varphi$ on the set of conjugacy classes of nontrivial elements on $\FN$ does not have a finite orbit.  Paralleling the result of Thurston about fibered 3--manifolds, combined results of Bestvina--Feighn and Brinkmann show that the semi-direct product using the automorphism $\Phi \in \Aut(\FN)$:
\begin{equation*}
\FN \rtimes_{\Phi} \ZZ = \I{ x_1, \ldots, x_N, t \mid t\inv x_{i} t = \Phi(x_i) }
\end{equation*}
is $\delta$--hyperbolic if and only if the outer automorphism class $[\Phi] \in \Out(\FN)$ is atoroidal~\cite{BF,brink}.  

Our main result is the analog of Ivanov's theorem in the setting of $\Out(\FN)$ corresponding to atoroidal elements.

\begin{alphathm}\label{th:alternative} 
Let $\calH$ be a subgroup of $\Out(\FN)$ where $N \geq 3$.  Either $\calH$ contains an atoroidal element or there exists a finite index subgroup $\calH'$ of $\calH$, and a nontrivial element $g\in\FN$ such that $\calH'[g]=[g]$.
\end{alphathm}

When $N = 2$ the theorem holds as well.  This follows as $\Out(F_{2})$ is naturally isomorphic to the extended mapping class group of a torus with a single boundary component and hence every subgroup has an index two subgroup that fixes the conjugacy class corresponding to the boundary component.

Essential to our proof of this theorem is the analog of Ivanov's theorem in the setting of $\Out(\FN)$ corresponding to fully irreducible elements as recently shown by Handel--Mosher~\cite{HM}.  Specifically, they prove that for a finitely generated subgroup $\calH < \Out(\FN)$, either $\calH$ contains a fully irreducible element or there exists a finite index subgroup $\calH'$ of $\calH$, and a proper free factor $A < \FN$ such that $\calH'[A]=[A]$.  The idea of their proof is similar in spirit to that of Ivanov.  If two elements in $\calH$ have sufficiently transverse finite orbits on the set of conjugacy classes of proper free factors, then some product of their powers is fully irreducible.  In this setting Handel--Mosher use the action on the space of laminations on $\FN$.  Later, Horbez generalized this result to all subgroups of $\Out(\FN)$ dropping the finitely generated assumption using the action of $\Out(\FN)$ on the free factor complex~\cite{Hshort}. 

Whereas Ivanov's theorem allows for repeated inward application to decompose a surface completely relative to the action of some subgroup $\calH < \Mod(S)$, the above stated version in $\Out(\FN)$ for fully irreducible elements does not.  The difference arises as if a subsurface is invariant, so is its complement, but if the conjugacy class of a proper free factor $A$ is invariant, there is no reason why there must be an invariant splitting $A \ast B$.  Handel--Mosher have extended their above mentioned result to give a complete decomposition of $\FN$ relative to the action of some finitely generated subgroup $\calH < \Out(\FN)$.  Specifically, they show that for any maximal $\calH$--invariant filtration $\emptyset = \calF_{0} \sqsubset \calF_{1} \sqsubset \cdots \sqsubset \calF_{k} = \{[\FN]\}$ of free factor systems, if the extension $\calF_{i-1} \sqsubset \calF_{i}$ is multi-edge, then there is an element $\varphi \in \calH$ which is fully irreducible with respect to this extension (see Section~\ref{sec:outf} for full details).  More recently, Horbez--Guirardel generalized this classification to all subgroups of $\Out(\FN)$ using the action of $\Out(\FN)$ on several hyperbolic complexes~\cite{un:GH}.

The proof of Theorem~\ref{th:alternative} builds on the above subgroup decomposition results.  The general strategy is to work from the bottom up: if $\calH$ contains an element whose restriction to $\calF_{i-1}$ is atoroidal either we find an element in $\calF_{i}$ whose orbit is finite, or we produce an element in $\calH$ whose restriction to $\calF_{i}$ is atoroidal.  Techniques and results from Handel--Mosher and Guirardel--Horbez take care of the case when $\calF_{i-1} \sqsubset \calF_{i}$ is a multi-edge extension.  The single-edge case requires a different approach.  Indeed,   
when $\calF_{i-1} \sqsubset \calF_{i}$ is a single-edge extension, the corresponding space of laminations is empty and so Handel--Mosher techniques do not apply.  On the other hand, trying to prove Theorem \ref{th:alternative} using solely by hyperbolic geometric methods is hopeless.  There are commuting non-atoroidal elements in $\Out(\FN)$ whose product is atoroidal (an example appears below) which implies there is no $\delta$--hyperbolic $\Out(\FN)$ complex whose loxodromic isometries are precisely atoroidal elements~\cite{TaylorMSRI}.

In order to deal with single-edge extensions, we use the space of geodesic currents $\Curr(\FN)$ (see Section~\ref{sec:currents} for full details).  This is the natural space for exhibiting that an element is atoroidal as it can be naturally viewed as the closure of the space of conjugacy classes in $\FN$.  Our main technical result, Theorem \ref{th:gns}, analyzes the dynamics of an element $\varphi \in \Out(\FN)$ that leaves invariant a co-rank $1$ free factor $A$ and whose restriction to $A$ is atoroidal.  If $\varphi$ is not atoroidal, we show that there are simplices $\Delta_{+}, \Delta_{-}$ in $\PCurr(\FN)$ and a counting current $[\eta_{g}]$ for which $\varphi$ has \emph{generalized north-south dynamics} with $\hDelta_{+} = \Cone(\Delta_{+},[\eta_{g}])$ and $\hDelta_{-} = \Cone(\Delta_{-},[\eta_{g}])$.  Specifically, points outside of a neighborhood of $\hDelta_{-}$ are moved by $\varphi$ into a neighborhood of $\Delta_{+}$ and vice versa for $\varphi\inv$ (see Figure~\ref{fig:gns set-up}).  This set-up is akin to the set-up for a nonatoroidal fully irreducible element (which necessarily is a pseudo-Anosov homeomorphism of a surface with one boundary component), where the fixed counting current corresponds to the boundary component of the associated surface~\cite[Theorem~B]{UyaNSD}.  This result is of independent interest as there is little known about the action of nonatoriodal elements on $\Curr(\FN)$ in general.  

A natural question is whether there is a stronger conclusion to Theorem~\ref{th:alternative}.  Precisely, is it the case that if $\calH < \Out(\FN)$ contains an atoroidal element, must it be that either $\calH$ is virtually cyclic or else contains a subgroup isomorphic to $F_{2}$ in which every nontrivial element is atoroidal?  The corresponding analog in the setting of $\Mod(S)$ is true and was shown by Ivanov~\cite{Iva}; the corresponding analog for fully irreducible elements in $\Out(\FN)$ is true and was shown by Kapovich--Lustig~\cite{KL5}.  In the present setting however, the stronger conclusion does not hold.  The key point, as it was for obstructing a $\delta$--hyperbolic complex whose loxodromic isometries are precisely the atoroidal elements, is that the centralizer of an atoroidal element is not virtually cyclic in general.  Indeed, if $\varphi \in \Out(F_{3})$ is atoroidal, then so is $\varphi * \varphi \in \Out(F_{3} * F_{3})$.  The subgroup $\calH = \I{\varphi * {\rm id}, {\rm id} * \varphi}$ is free abelian of rank $2$ and contains an atoroidal element.

In light of the above discussion, one might conjecture that if $\calH < \Out(\FN)$ contains an atoroidal element $\varphi$, then either $\calH$ \emph{virtually centralizes} $\varphi$: for all $h \in \calH$, there is an $n  > 0$ such that $h\varphi^{n} = \varphi^{n} h$ or $\calH$ contains a subgroup isomorphic to $F_{2}$ in which every nontrivial element is atoroidal.  However, even this weaker statement is not true.  For example, take atoroidal elements $\varphi, \psi \in \Out(F_{3})$ such that $\I{\varphi,\psi} \cong F_{2}$ and consider the subgroup $\calH = \I{\varphi*\varphi,\varphi*\psi} \subset \Out(F_{6})$.  Any non-trivial element of $\calH$ is of the form $\varphi^{n} * \omega$ where $n \in \ZZ$ and $\omega \in \I{\varphi,\psi}$ is non-trivial.  In particular $\calH$ does not virtually centralize any of its non-trivial elements.  However, given any two elements $\theta_{1}, \theta_{2}  \in \calH$, we have $\theta_{1} = \varphi^{n_{1}} * \omega_{1}$ and $\theta_{2} = \varphi^{n_{2}} * \omega_{2}$ and thus we find that $\theta_{1}^{n_{2}}\theta_{2}^{-n_{1}} = {\rm id} * \omega_{1}^{n_{2}}\omega_{2}^{-n_{1}}$ which is not atoroidal.  Therefore $\I{\theta_{1},\theta_{2}}$ is not purely atoroidal. 

The right characterization is the following statement. 

\begin{alphathm}\label{th:noTits} 
Let $\calH < \Out(\FN)$ be a subgroup which contains an atoroidal element $\varphi$.  Then, $\calH$ contains a purely atoroidal free subgroup if and only if the restriction of $\calH$ to each minimal $\calH$--invariant free factor is not virtually cyclic. 
\end{alphathm}

\begin{proof} 
The ``if'' direction follows from \cite[Lemma 4.3]{U3}. For the other direction, let $A<\FN$ be a minimal $\calH$--invariant free factor such that the restriction of $\calH$ to $A$ is virtually cyclic (see Section~\ref{subsec:relative} for definitions). The proof of \cite[Lemma 4.3]{U3} implies that the restriction of each element in $\calH$ to $A$ has a power which is equal to a power of the restriction of $\varphi$ to $A$. Now assume that $\calH$ contains a subgroup isomorphic to $F_{2}$, generated by $\psi_1, \psi_2$. By above observation, there exist nonzero integers $n_{1},n_{2}, k$ such that $\psi_1\big|_{A}^{n_{1}}=\varphi\big|_{A}^{k}=\psi_2\big|_{A}^{n_{2}}$.  Then the element $\psi_1^{n_{2}}\psi_2^{-n_{1}}\in \I{\psi_1,\psi_2}$ fixes each element in $A$ and hence is not atoroidal.  Thus the subgroup $\I{\psi_{1},\psi_{2}}$ is not purely atoroidal.
\end{proof}



\medskip

\noindent {\bf Organization of paper.} Section~\ref{sec:outf} reviews the theory of outer automorphisms needed.  In particular, the notions of free factor systems, the Handel--Mosher subgroup decomposition and train tracks are recalled.  Definitions of geodesic currents are presented in Section \ref{sec:currents}.  As mentioned above, we deal separately with multi-edge and single-edge extensions.  Section~\ref{sec:push past multi edge} shows how to apply the results of Handel--Mosher and Guirardel--Horbez to push past multi-edge extensions.  The main technical result, that of generalized North-South dynamics for co-rank $1$ atoroidal elements, constitutes the majority of Section~\ref{sec:single-edge}.   In Section~\ref{sec:push past one edge}, we show how to apply this result to push past single-edge extensions.  Lastly, in Section~\ref{sec:proof}, we combine the above two cases to complete the proof of Theorem~\ref{th:alternative}.

\medskip 

\noindent {\bf Acknowledgements.} The authors thank Camille Horbez for telling them about his upcoming work with Guirardel \cite{un:GH} and useful discussions. Second author is grateful to Jon Chaika for illuminating discussions regarding ergodic theory.  The authors also thank the anonymous referee for a careful reading and several helpful suggestions.


\section{Outer automorphisms and train tracks}\label{sec:outf}

In this section we collect definitions and some of the fundamental results regarding $\Out(\FN)$ we use in the sequel.


\subsection{Graphs, maps and markings}\label{subsec:graphs}

A \emph{graph} $G$ is a $1$--dimensional cell complex.  The $0$--cells of $G$ are called \emph{vertices}, and the $1$--cells of $G$ are called \emph{\textup{(}topological\textup{)} edges}. We denote the set of vertices by $VG$ and the set of edges by $E_{top}G$. Identifying the interior of each topological edge $e\in E_{top}G$ with the open interval $(0,1)$ we get exactly two orientations on $e$.  The set of oriented edges of $G$ is denoted by $EG$. For each edge $e\in E_{top}G$, we choose a \emph{positive} orientation for $e$, and denote the set of positively oriented edges by $E^{+}G$. Given an oriented edge $e \in EG$, the edge with the opposite orientation is denoted by $e^{-1}$. Furthermore, we denote the initial point of the oriented edge $e$ by $o(e)$ and the terminal point by $t(e)$.

Of particular importance is the $N$--rose, denoted by $R_{N}$, which is the graph with a single vertex $v$ and $N$ edges.  We fix an isomorphism $\FN \cong \pi_{1}(R_{N},v)$ which we will use implicitly throughout.  Using this isomorphism, homotopy equivalences of $R_{N}$ determine outer automorphisms of $\FN$ and vice versa.   
 
An \emph{edge path} $\gamma$ of length $n$ is a concatenation $\gamma=e_1e_2\ldots e_n$ of oriented edges in $G$ such that $t(e_i)=o(e_{i+1})$ for all $i=1, \ldots, n-1$.  The length of a path is denoted by $|\gamma|$.  The edge path $\gamma$ as above is called \emph{reduced} if $e_{i}\neq e^{-1}_{i+1}$ for all $i=1, \ldots, n-1$.  Further, a reduced edge path $\gamma=e_1e_2 \ldots e_n$ is called  \emph{cyclically reduced} if $t(e_n)=o(e_1)$ and $e_{n}\neq e_1^{-1}$. For any edge path $\gamma$, there is a unique reduced edge path $[\gamma]$ homotopic to $\gamma$ rel endpoints.  

A (topological) \emph{graph map}  $f \from G_{0} \to G_{1}$ is a homotopy equivalence where: 
\begin{itemize}
\item $f(VG_{0}) \subseteq VG_{1}$; and 
\item the restriction of $f$ to interior of an edge is an immersion. 
\end{itemize}
These conditions imply that for each oriented edge $e\in EG_{0}$, the image $f(e)$ determines a reduced edge path.  A graph map $m \from R_{N} \to G$ is called a \emph{marking of $G$}.  Suppose $m \from R_{N} \to G$ is a marking and fix a graph map $m' \from G \to R_{N}$ that is homotopy inverse to $m$.  We say that a graph map $f \from G \to G$ is a \emph{topological representative} of the outer automorphism $\varphi\in\Out(\FN)$ if the outer automorphism determined by the homotopy equivalence $m' \circ f \circ m \from R_{N} \to R_{N}$ is $\varphi$.

A \emph{filtration} for a topological representative $f \from G\to G$ is an increasing sequence of $f$--invariant subgraphs $\emptyset=G_0\subset G_1\subset \cdots \subset G_{\ell} =G$. The \emph{r{\rm th}-stratum} in this filtration, denoted by $H_r$, is the closure of $G_r - G_{r-1}$.  Associated to each stratum $H_r$ there is a square matrix whose row and columns are indexed by the edges in $H_{r}$ called the \emph{transition matrix} $M_{r}$, which is non-negative and has integer entries.  The $ij$th entry of $M_{r}$ records the number of times the reduced path $f(e_i)$ crosses the edge $e_j$ or the edge $e_{j}^{-1}$.

Recall, a non-negative square matrix $M$ is called \emph{irreducible} if for each $i,j$, there exists $p = p(i,j)$ such that $M^{p}_{ij}>0$. 
We say that the stratum $H_r$ is \emph{irreducible} if the associated transition matrix $M_r$ is irreducible.  If $M_r$ is irreducible then it has a unique eigenvalue $\lambda_{r} \geq 1$ called the \emph{Perron-Frobenius} eigenvalue, for which the associated eigenvector is positive. We say that $H_r$ is an \emph{exponentially growing \textup{(}EG\textup{)} stratum} if $\lambda_{r} > 1$. We say that $H_r$ is a \emph{non-exponentially growing \textup{(}NEG\textup{)} stratum} if $\lambda_{r} = 1$. Finally, we say that $H_r$ is a \emph{zero stratum} if $M_r$ is the zero matrix. 


\subsection{Free factor systems and geometric realizations}\label{subsec:free factor}

A \emph{free factor} $\A<\FN$ is a subgroup of $\FN$ such that $\FN = \A \ast \B$ where $\B<\FN$ is a (possibly trivial) subgroup of $\FN$.  A free factor is called \emph{proper} if it is neither the trivial subgroup nor $\FN$.  The conjugacy class of a free factor $\A$ is denoted by $[\A]$.  A \emph{free factor system} $\calF=\{[\A^{1}], \ldots, [\A^{k}]\}$ is a collection of conjugacy classes of free factors of $\FN$ such that 
\begin{equation*}
\FN = \A^{1} \ast \A^2 \ast \cdots \ast \A^k \ast \B
\end{equation*}
for some representatives $\A^{i}$ of $[\A^{i}]$ and for some (possibly trivial) subgroup $\B<\FN$.  

A subgraph $K \subseteq G$ of a marked graph $G$ determines a free factor system $\calF(K)$ of $\FN$ in the following way.  Enumerate the non-contractible components of $K$ by $C_{1}, \ldots, C_{k}$, fix vertices $v_{i} \in C_{i}$ and edge paths $\gamma_{i}$ from $v_{i}$ to $v$ (some arbitrary vertex of $G$).  These paths induce inclusions $\pi_{1}(C_{i},v_{i}) \to \pi_{1}(G,v)$.  The conjugacy classes of the images do not depend on the $v_{i}$'s nor the $\gamma_{i}$'s and the collection $\{[\pi_{1}(C_{1},v_{1})] , \ldots, [\pi_{1}(C_{k},v_{k})] \}$ is a free factor system of $\pi_{1}(G,v)$.  Using the marking of $G$ we obtain a free factor system $\calF(K)$ of $\FN$.

There is a natural partial order among free factor systems.  Given free factor systems $\calF_{0} =\{[\A^1], \ldots, [\A^{k}]\}$ and $\calF_{1}=\{[\B^1], \ldots, [\B^\ell]\}$ we say that $\calF_{0}$ is \emph{contained} in $\calF_{1}$ (or $\calF_{1}$ is an \emph{extension} of $\calF_{0}$) and write $\calF_{0} \sqsubset \calF_{1}$ if for each $i=1,\ldots, k$, there exist $j\in \{1,\ldots, \ell\}$ and $g\in\FN$ such that $\A^{i}$ is a subgroup of $g\B^{j}g^{-1}$.  An extension $\calF_{0} \sqsubset \calF_{1}$ is called a \emph{single-edge extension} if there exists a marked graph $G$ with subgraphs $G_{0},G_{1}$ such that $\calF(G_{0}) = \calF_{0}$, $\calF(G_{1}) = \calF_{1}$ and $G_{1} - G_{0}$ is a single edge.  Otherwise, $\calF_{0} \sqsubset \calF_{1}$ is called a \emph{multi-edge extension}.  There are three types of single-edge extensions. In a \emph{circle extension} $G_1$ is obtained from $G_{0}$ by adding a disjoint loop edge.  In a \emph{barbell extension}, a single edge is attached to two distinct components of $G_{0}$.  Finally, attaching an edge to the same component of $G_{0}$ gives a \emph{handle extension}. 

 A \emph{filtration of $\FN$ by free factor systems} is an ascending sequence $\emptyset = \calF_{0} \sqsubset \calF_{1} \sqsubset \cdots \sqsubset \calF_{k} = \{[ \FN] \}$ of free factor systems.  We say that a filtration $\emptyset = \calF_{0} \sqsubset \calF_{1} \sqsubset \cdots \sqsubset \calF_{k} = \{[ \FN] \}$ is \emph{realized} by the filtration $\emptyset=G_0\subset G_1\subset \cdots \subset G_{\ell}=G$ of a marked graph $G$ if for each $i = 1, \ldots, k$ there is an $j \in \{1, \ldots, \ell \}$ such that $\calF_{i} = \calF(G_{j})$.


\subsection{Relative outer automorphisms}\label{subsec:relative}

Outer automorphisms act on the set of conjugacy classes of free factors and on the set of free factor systems.  An element $\varphi \in \Out(\FN)$ is \emph{irreducible} if there does not exist a proper free factor system $\calF$ such that $\varphi\calF = \calF$; $\varphi$ is \emph{fully irreducible} if $\varphi^{p}$ is irreducible for all $p \geq 1$.  If $\calF_{0} \sqsubset \calF_{1}$ is a $\varphi$--invariant extension, we say $\varphi$ is \emph{irreducible with respect to $\calF_{0} \sqsubset \calF_{1}$} if there does not exist a $\varphi$--invariant factor free system $\calF \neq \calF_{0},\calF_{1}$ such that $\calF_{0} \sqsubset \calF \sqsubset \calF_{1}$; $\varphi$ is \emph{fully irreducible with respect to $\calF_{0} \sqsubset \calF_{1}$} if $\varphi^{p}$ is irreducible with respect to $\calF_{0} \sqsubset \calF_{1}$ for all $p \geq 1$.  Irreducibility is equivalent to irreducibility with respect to the extension $\{[ \I{1} ]\} \sqsubset \{[\FN]\}$. 

We usually work with elements in the finite-index subgroup:
\begin{equation*}
\IA_{N}(\ZZ/3) = {\rm ker}\bigl(\Out(\FN) \to \Aut(H_{1}(\FN,\ZZ/3))\bigr).
\end{equation*}
For elements in this subgroup, periodic phenomena become fixed. In particular, Handel--Mosher showed that for any $\varphi \in \IA_{N}(\ZZ/3)$: 
\begin{enumerate}

\item any $\varphi$--periodic free factor system in $\FN$ is fixed by $\varphi$ \cite[Theorem~3.1]{HMpart2}; and
\item any $\varphi$--periodic conjugacy class in $\FN$ is fixed by $\varphi$ \cite[Theorem~4.1]{HMpart2}.

\end{enumerate}
Thus irreducible and fully irreducible are identical notions in this subgroup.

Of central importance to the theory of relative outer automorphisms is the Handel--Mosher Subgroup Decomposition Theorem. 

\begin{theorem}[{\cite[Theorem~D]{HMIntro}}]\label{th:HMdecomp}
Given a finitely generated subgroup $\calH < \IA_{N}(\ZZ/3)$ and a maximal $\calH$--invariant filtration $\emptyset = \calF_{0} \sqsubset \calF_{1} \sqsubset \cdots \sqsubset \calF_{k} = \{[ \FN] \}$, for each $i = 1,\ldots, k$ such that $\calF_{i-1} \sqsubset \calF_{i}$ is a multi-edge extension, there is an element $\varphi_{i} \in \calH$ that is irreducible with respect to $\calF_{i-1} \sqsubset \calF_{i}$.
\end{theorem}

\begin{remark}\label{rem:HMdecomp}
In fact, a single $\varphi \in \calH$ satisfies the conclusion of the theorem~\cite[Theorem~6.6]{CU}.
\end{remark}

We denote the stabilizer in $\Out(\FN)$ of a free factor system $\calF$ of $\FN$ by $\Out(\FN;\calF)$.  If $\calF = \{[A]\}$, we usually write $\Out(\FN;A)$ for this subgroup.  

Suppose $A < \FN$ is a free factor and $\varphi \in \Out(\FN;A)$.  Then there is an automorphism $\Phi \in \varphi$ such that $\Phi(A) = A$.  The outer automorphism class of the restriction of $\Phi$ to $A$ is the same for any representative of $\varphi$ that fixes $A$, we denote the resulting outer automorphism by $\varphi\big|_{A} \in \Out(A)$.  Moreover, the assignment $\varphi \mapsto \varphi\big|_{A}$ is a homomorphism from $\Out(\FN;A)$ to $\Out(A)$~\cite[Fact~1.4]{HMpart1}.  

If $\varphi \in \Out(\FN)$ fixes each element of a free factor system $\calF = \{[A^{1}],\ldots,[A^{k}]\}$ then we write $\varphi\big|_{\calF}$ to refer to the collection of maps $\left\{\varphi\big|_{A^{1}},\ldots,\varphi\big|_{A^{k}}\right\}$.  This happens in particular when $\varphi \in \IA_{N}(\ZZ/3) \cap \Out(\FN;\calF)$.  If we say $\varphi\big|_{\calF}$ has some property (e.g. is atoroidal), we mean each of the maps $\varphi\big|_{A^{i}}$ has this property.  


\subsection{Train tracks and CTs}\label{subsec:train-track}

Train track maps are a type of graph map with certain useful features that were first introduced by Bestvina--Handel in order to study the dynamics of irreducible outer automorphisms of $\FN$.  Not every outer automorphism is represented by a train track map, but they can be represented by a generalization called a \emph{relative train track map} \cite{BH92}.  Since their original construction, train track maps have been improved upon giving finer control over certain aspects of the maps.  For our purpose, we will work with a \emph{completely split train track map \textup{(}CT\textup{)}} introduced by Feighn--Handel \cite{FH}.  The definition of a $\CT$ is rather long and technical and so after giving the definition of a relative train track map below (Definition~\ref{def:relative train track}), we will only state the relevant properties of a CT needed in the sequel (Lemma~\ref{CT-properties}).  We also quote the key result that after passing to a power, every outer automorphism can be represented by a CT (Theorem~\ref{th:CT exist}).

A graph map $f \from G\to G$ induces a \emph{derivative map} $Df \from EG\to EG$ on the set of oriented edges by setting $Df(e)$ equal to the first edge in the edge path $f(e)$.  A \emph{turn} in $G$ is an unordered pair $(e_1, e_2)$ of oriented edges in $G$ where $o(e_{1}) = o(e_{2})$.  A turn $(e_{1},e_{2})$ is called \emph{degenerate} if $e_1=e_2$, otherwise it is called \emph{non-degenerate}. A turn $(e_{1},e_{2})$ is called \emph{illegal} if its image $\bigl((Df)^{k}(e_1), (Df)^{k}(e_2)\bigr)$ under an iterate of the derivative map is degenerate for some $k \geq 1$, otherwise it is called \emph{legal}. An edge path $e_1e_2\ldots e_n$ is called legal if each turn $(e_i^{-1}, e_{i+1})$ for $i=1, \ldots, n-1$ is legal. 

Suppose $\emptyset=G_0\subset G_1\subset \cdots\subset G_\ell=G$ is a filtration of the map $f$.  We say that a turn $(e_1,e_2)$ is contained in the stratum $H_r$ if both edges $e_1, e_2$ are in $EH_r$.  An edge path $\gamma$ is called \emph{$r$--legal}, if every turn in $\gamma$ that is contained in $H_r$ is legal. A \emph{connecting path} for $H_r$ is a nontrivial reduced path $\gamma$ in $G_{r-1}$ whose endpoints are in $G_{r-1} \cap H_r$; it is \emph{taken} if it is the subpath of $[f^{k}(e)]$ for some edge $e$ that belongs to an irreducible stratum.  

\begin{definition}\label{def:relative train track} 
A topological graph map $f:G\to G$ equipped with a filtration $\emptyset=G_0\subset G_1\subset \cdots\subset G_\ell=G$ is called a \emph{relative train track map} if for each exponentially growing stratum $H_r$ the following hold: 

\begin{enumerate}
\item for each edge $e\in EH_r$, $(Df)^{k}(e)\in EH_{r}$ for all $k \geq 1$;
\item for each connecting path $\gamma$ for $H_r$, the path $[f(\gamma)]$ is also a connecting path for $H_r$; and 
\item if $\gamma$ is $r$--legal, then so is $[f(\gamma)]$. 
\end{enumerate}
\end{definition}

The notion of a \emph{geometric stratum} for a relative train track map was introduced and studied by Bestvina--Feighn--Handel \cite{BFH00}, and studied extensively by Handel--Mosher in the CT setting~\cite{HMpart1}.  Suppose $\emptyset = G_{0} \subset G_{1} \subset \cdots \subset G_{\ell} = G$ is a filtration for a relative train track map $f \from G \to G$.  A stratum $H_r$ is called \emph{geometric} if there exist a compact surface $S$ with $k+1$ boundary components $\alpha_0, \alpha_1,\ldots, \alpha_k$ and a pseudo-Anosov homeomorphism $h \from S \to S$ with the following properties.

\begin{itemize}
\item The homeomorphism $h$ extends to a homotopy equivalence $h \from S\cup G_{r-1} \to S \cup G_{r-1}$ where $S$ is attached to $G_{r-1}$ by attaching the boundary components $\alpha_{1},\ldots,\alpha_{k}$ to $k$ circuits in $G_{r-1}$.

\item There is an embedding $G_{r} \hookrightarrow S \cup G_{r-1}$ that restricts to the identity on $G_{r-1}$ and a deformation retraction $d \from S \cup G_{r-1} \to G_{r}$ such that $fd \simeq dh$.  
\end{itemize}

We can extend this notion to subgroups of $\Out(\FN)$.  Suppose $\calH$ is a subgroup of $\Out(\FN)$ and $\calF_{0} \sqsubset \calF_{1}$ is a multi-edge extension invariant under $\calH$.  We say the extension is \emph{geometric} if for each $\varphi \in \calH$ there is a relative train track map $f \from G \to G$ with a filtration $\emptyset = G_{0} \subset G_{1} \subset \cdots \subset G_{\ell} = G$ realizing the filtration for $\FN$ such that the stratum $H_{r}$ is geometric where $\calF_{0} = \calF(G_{r-1})$ and $\calF_{1} = \calF(G_{r})$, without the assumption that the associated homeomorphism $h \from S \to S$ is pseudo-Anosov.  We call $S$ a \emph{geometric model} for $\varphi$.

The following lemma summarizes the key additional properties of CT maps that we will use.  To state the first of these properties, we need the following definition.  A path $\rho$ in $G$ is a \emph{Nielsen path} if $[f^{k}(\rho)]=\rho$ for some $k \geq 1$; it is an \emph{indivisible Nielsen path} if further it does not split as the concatenation of two non-trivial Nielsen paths. 

\begin{lemma}\label{CT-properties}  
Suppose $f \from G \to G$ is a CT map with filtration $\emptyset = G_{0} \subset G_{1} \subset \cdots \subset G_{\ell} = G$.    
\begin{enumerate}

\item If $H_r$ is a non-geometric EG stratum, then there does not exist a closed Nielsen path $\rho \subset G_{r}$ that intersects $H_{r}$ nontrivially \textup{(}\cite[Corollary~4.19~eg(ii)]{FH} and \cite[Fact~1.42~(1b)]{HMpart1}\textup{)}.\label{CT-properties:EG}

\item If $H_r$ is an NEG stratum, then $H_r$ consists of a single edge $e$.  Furthermore, either $e$ is fixed, or $f(e) = e\gamma$ where $\gamma$ is a nontrivial cyclically reduced path in $G_{r-1}$ \textup{(}\cite[Lemma~4.21]{FH}\textup{)}.\label{CT-properties:NEG}
\end{enumerate}
\end{lemma}

The edge $e$ of an NEG stratum is called a \emph{fixed edge} if $f(e)=e$, a \emph{linear edge} if $f(e)=e  \rho$ where $\rho$ is a nontrivial Nielsen path, and a \emph{superlinear edge} otherwise.  We conclude this section by stating the theorem providing the existence of CT maps.

\begin{theorem}[{\cite[Theorem 4.28, Lemma 4.42]{FH}}]\label{th:CT exist} 
There exist a constant $M = M(N) \ge 1$ such that for any $\varphi \in \Out(\FN)$, and any nested sequence $\calC$ of $\varphi^{M}$-invariant free factor systems, there exists a CT map $f \from G \to G$ that represents $\varphi^{M}$ and realizes $\calC$. 
\end{theorem}


\section{Geodesic currents}\label{sec:currents}

The way we demonstrate that an element of $\Out(\FN)$ is atoroidal is by showing that it acts on a certain space without a periodic orbit.  The space we consider is the space of geodesic currents, which naturally contains the set of conjugacy classes of nontrivial elements of $\FN$.  We describe this space and its key features in this section.  More details can be found in~\cite{Ka2}.

Let $\partial \FN$ denote the \emph{Gromov boundary} of $F_N$.  The \emph{double boundary} of $\FN$ is defined to be the set:
\begin{equation*}
\partial^{2} \FN=(\partial \FN\times \partial \FN \setminus \Delta) /\sim
\end{equation*}
where $\sim$ is the flip relation $(x,y) \sim (y,x)$, and $\Delta$ is the diagonal.  This set is naturally identified with the set of unoriented bi-infinite geodesics in $\widetilde{R}_{N}$, the universal cover of $R_{N}$.  The group $\FN$ acts on itself by left multiplication, which induces an action of $\FN$ on both $\partial \FN$ and $\partial^{2}\FN$. 

A \emph{geodesic current on $\FN$} is a non-negative Radon measure on $\partial^{2}\FN$ that is invariant under the action of $\FN$.  The space of geodesic currents on $\FN$, denoted by $\Curr(\FN)$, is equipped with the weak-* topology.  We give more specifics about the topology later. 

The following construction is the most natural example of a geodesic current. Let $g \in \FN$ be a nontrivial element that is not a proper power, i.e.,~$g \neq h^{k}$ for some $h\in\FN$, and $k>1$. Let $(g^{-\infty}, g^{\infty})$ be the unoriented bi-infinite geodesic labeled by $g$'s. For any such $g$ we define the \emph{counting current} $\eta_{g} \in \Curr(\FN)$ as follows. If $S\subset\partial^{2}\FN$ is a Borel subset we set:
\begin{equation*}
\eta_{g}(S) = \#\abs{S \cap \FN(g^{-\infty},g^{\infty})}.
\end{equation*}
This definition does not depend on the representative of the conjugacy class $[g]$ of $g$, so we will use $\eta_{[g]}$ and $\eta_{g}$ interchangeably.  For an arbitrary $g$, we write $g=h^{k}$ where $h$ is not a proper power and define $\eta_{g}=k\eta_{h}$.  The set of scalar multiples of all counting currents are called \emph{rational currents}.  An important fact about rational currents is that they form a \emph{dense} subset of $\Curr(\FN)$ \cite{Bosurvey, Ka2, Martin}

The group $\Aut(\FN)$ acts by homeomorphisms on $\Curr(\FN)$ as follows.  An automorphism $\Phi\in \Aut(\FN)$, extends to a homeomorphism of both $\partial\FN$ and $\partial^{2}\FN$ which we still denote by $\Phi$, and for $\mu \in \Curr(F_{n})$ we define:
\begin{equation*}
(\Phi\mu)(S)=\mu(\Phi^{-1}(S))
\end{equation*}
for any Borel subset $S$ of $\partial^{2}\FN$.  The $\FN$--invariance of the measure implies that the group $\Inn(\FN)$ of inner automorphisms acts trivially, hence we obtain an action of $\Out(\FN)=\Aut(\FN)/\Inn(\FN)$ on $\Curr(\FN)$. On the level of conjugacy classes one can easily verify that $\varphi\eta_{[g]}=\eta_{\varphi [g]}$.

The space $\PCurr(\FN)$ of \emph{projectivized geodesic currents} is defined as the quotient of $\Curr(\FN)-\{0\}$ where two currents are deemed equivalent if they are positive scalar multiples of each other. The space $\PCurr(\FN)$ endowed with the quotient topology is compact \cite{Bosurvey, Ka2}. 
Furthermore, setting $\varphi[\mu]=[\varphi\mu]$ gives a well defined action of $\Out(\FN)$ on $\PCurr(\FN)$. 

We will now give more specifics about the topology on $\Curr(\FN)$.  Let $m \from  R_{N} \to G$ be a marking.   Lifting $m$ to the universal covers, we get a quasi-isometry $\tilde{m} \from \widetilde{R}_N \to \widetilde{G}$ and a homeomorphism $\tilde{m} \from \partial \FN \to \partial \widetilde{G}$.  Given a reduced edge path $\tilde\gamma$ in $\tilde{G}$ the \emph{cylinder set} of $\tilde\gamma$ is defined as 
\begin{equation*}
Cyl_m(\tilde\gamma)=\left\{(\xi_1,\xi_2)\in\partial^2F_N\mid \tilde\gamma\subset[\tilde{m}(\xi_1), \tilde{m}(\xi_2)]\right\},
\end{equation*}
where $[\tilde{m}(\xi_1), \tilde{m}(\xi_2)]$ is the bi-infinite geodesic from $\tilde{m}(\xi_1)$ to $\tilde{m}(\xi_2)$ in $\tilde{G}$ and containment is for either orientation. 

Let $\gamma$ be a reduced edge path in $G$ and let $\tilde\gamma$ be a lift of $\gamma$ to $\widetilde{G}$.  We define the number of \emph{occurrences} of $\gamma$ in $\mu$ as
\begin{equation*}
\I{\gamma,\mu}_m =\mu(Cyl_m(\tilde\gamma)).
\end{equation*}
As $\mu$ is invariant under the action of $\FN$, the quantity $\mu(Cyl_{m}(\tilde\gamma))$ does not depend on the choice of the lift $\tilde\gamma$ of $\gamma$.  Hence, $\I{\gamma,\mu}_m$ is well defined.  The marked graph will always be clear from the context and in what follows we drop the letter $m$ from the notation and use $Cyl(\tilde\gamma)$ and $\I{\gamma, \mu}$.  

Cylinder sets form a subbasis for the topology of the double boundary $\partial^{2}\FN$ and play an important role in the topology of currents. In \cite{Ka2}, it was shown that a geodesic current is uniquely determined by the set of values $\{\I{\gamma ,\mu}\}_{\gamma}$ as $\gamma$ varies over the set of all reduced edge paths in $G$. 

Furthermore, defining the \emph{simplicial length of a current $\mu$} to be $\wght{\mu} = \sum_{e \in E^{+}G} \I{e,\mu}$ we have the following characterization of limits in $\PCurr(\FN)$.  

\begin{lemma}[{\cite[Lemma~3.5]{Ka2}}]\label{lem:topology}
Suppose $([\mu_n]) \subset \PCurr(\FN)$ is a sequence and $[\mu] \in \PCurr(\FN)$.  Then \begin{equation*}
\lim_{n\to\infty}[\mu_{n}]=[\mu] \text{ \, if and only if \, } \lim_{n\to\infty} \dfrac{\langle \gamma, \mu_{n}\rangle}{|\mu_{n}|}=\dfrac{\langle \gamma, \mu\rangle}{|\mu|}
\end{equation*}
for each reduced edge path $\gamma$ in $G$. 
\end{lemma}

The value $\wght{\mu}$ does depend on the marked graph, but as before, the marked graph will always be clear from the context and so we omit it from the notation.  It follows immediately from Lemma~\ref{lem:topology} that the occurrence function $\mu \mapsto \I{\gamma,\mu}$ and the simplicial length function $\mu \mapsto \wght{\mu}$ are continuous and linear on $\Curr(\FN)$~\cite[Proposition~5.9]{Ka2}.

Given a free factor $A < \FN$, let $\iota \from A \to \FN$ be the inclusion map. There is a canonical $A$--equivariant embedding  $\partial A \subset \partial\FN$ which induces an $A$--equivariant embedding $\partial^{2}A\subset \partial^{2}\FN$.  Let $\Curr(A)$ and $\Curr(\FN)$ be the corresponding spaces of currents.  There is a natural inclusion $\iota_{A} \from \Curr(A) \to \Curr(\FN)$ defined by \emph{pushing the measure forward} via the $\FN$ action such that for each $g\in A$ we have $\iota_{A}(\eta_{g})=\eta_{\iota(g)}$, see \cite[Proposition-Definition 12.1]{Ka2}.


\section{Pushing past multi-edge extensions}\label{sec:push past multi edge}

As stated in the introduction, the strategy for proof of Theorem~\ref{th:alternative} is to work from the bottom up using a maximal $\calH$--invariant filtration $\emptyset = \calF_{0} \sqsubset \calF_{1} \sqsubset \cdots \sqsubset \calF_{k} = \{[ \FN] \}$.  Assuming that there is an element $\varphi \in \calH$ such that $\varphi\big|_{\calF_{i-1}}$ is atoroidal, we either find a nontrivial element $g \in \FN$ whose conjugacy class is fixed by a finite index subgroup of $\calH$, or in the absence of such an element, we produce an element $\hat\varphi \in \calH$ such that $\hat\varphi\big|_{\calF_{i}}$ is atoroidal.  

There are two cases depending on whether the extension $\calF_{i-1} \sqsubset \calF_{i}$ is multi-edge or single-edge.  In this section we deal with the multi-edge case; the single-edge case takes up Section~\ref{sec:push past one edge}.  

The multi-edge case follows from recent work of Handel--Mosher and Guirardel--Horbez.  We collect these results here and show how they apply to this setting.

\begin{theorem}\label{th:multi edge dichotomy} 
Suppose $\calH < \IA_{N}(\ZZ/3) < \Out(\FN)$.  Let $\calF_{0} \sqsubset \calF_{1}$ be an $\calH$--invariant multi-edge extension, and assume that $\calH$ contains an element which is fully irreducible with respect to the extension $\calF_{0} \sqsubset \calF_{1}$.  Then one of the following holds. 
\begin{enumerate} 
\item $\calH$ contains an element $\psi$ which is fully irreducible and non-geometric relative to $\calF_{0} \sqsubset \calF_{1}$(\cite[Proposition~2.2 and 2.4]{HMpart4}); or 
\item there is a common geometric model for all $\varphi \in \calH$ and hence every element of $\calH$ fixes the conjugacy class corresponding to a boundary curve (\cite[Theorem~J]{HMpart4}). 
\end{enumerate}
\end{theorem}

When $\calF_{0}=\emptyset$, the above theorem was originally proved by the second author~\cite{UyaNSD}. The general case above is also proved by Guirardel--Horbez 
using the action of the relative outer automorphism group on a $\delta$--hyperbolic complex which is a relative version of Dowdall--Taylor's co-surface graph~\cite{DTcosurface}.  The existence and relevant properties of this complex, which we will also need, is the following.

\begin{theorem}\cite[Theorem 4.2]{un:GH}\label{th:relativecosurface} 
Suppose $\calF \sqsubset \{[\FN]\}$ is a multi-edge extension.  There exist a $\delta$--hyperbolic graph $\zF$ 
with an isometric $\Out(\FN;\calF)$ action so that an element $\varphi\in\Out(\FN;\calF)$ acts as a hyperbolic isometry of $\zF$ if and only if $\varphi$ is fully irreducible and non-geometric relative to $\calF \sqsubset \{[\FN]\}$.  
\end{theorem}

As a consequence of Theorem~\ref{th:multi edge dichotomy}, when considering the multi-edge extension $\calF_{i-1} \sqsubset \calF_{i}$ which is part of a maximal $\calH$--invariant filtration, if there does not exist a nontrivial element $g \in \FN$ whose conjugacy class is in $\calF_{i}$ and is fixed by a finite index subgroup of $\calH$, then there is a fully irreducible and non-geometric element $\varphi$ relative to $\calF_{i-1} \sqsubset \calF_{i}$.  Assuming $\varphi\big|_{\calF_{i-1}}$ is atoroidal, so is $\varphi\big|_{\calF_{i}}$ as the next lemma states, allowing us to push past a multi-edge extension.    

\begin{lemma}\label{co:non-geometric atoroidal}
Suppose $\varphi \in \Out(\FN)$ is fully irreducible and non-geometric with respect to the extension $\calF_{0} \sqsubset \calF_{1}$ and the restriction of $\varphi$ to $\calF_{0}$ is atoroidal.  Then the restriction of $\varphi$ to $\calF_{1}$ is atoroidal too.  
\end{lemma}

\begin{proof} 
This is a straightforward consequence of Lemma \ref{CT-properties}\eqref{CT-properties:EG}.  Indeed, let $f \from G \to G$ be a CT map that represents $\varphi^{M}$ and realizes $\calC = (\calF_{0},\calF_{1})$, where $M$ is the constant from Theorem~\ref{th:CT exist}.  Assume $M$ is so that $\varphi^{M} \in \IA_{N}(\ZZ/3)$.  Let $H_r$ be the stratum corresponding to the extension $\calF_{0}\sqsubset\calF_{1}$, i.e., $\calF_{0} = \calF(G_{r-1})$, $\calF_{1} = \calF(G_{r})$ and $H_r=\overline{G_r - G_{r-1}}$. 

Any $\varphi$--periodic conjugacy class contained in $\calF_{1}$ is represented by a closed Nielsen path $\rho \subset G_{r}$.  As $H_{r}$ is a non-geometric EG stratum, Lemma~\ref{CT-properties}\eqref{CT-properties:EG} implies that $\rho \subset G_{r-1}$, which contradicts the assumption that $\varphi\big|_{\calF_{0}}$ is atoroidal.
\end{proof} 

Combining the Handel--Mosher Subgroup Decomposition Theorem (Theorem~\ref{th:HMdecomp}) with Theorems~\ref{th:multi edge dichotomy} and \ref{th:relativecosurface}, we get the following corollary which will be required when pushing past single-edge extensions.

\begin{corollary}\label{co:HM-simultaneous}
Suppose $\calH < \IA_{N}(\ZZ/3) < \Out(F_N)$.  Let \[\emptyset = \calF_{0} \sqsubset \calF_{1} \sqsubset \cdots \sqsubset \calF_{k} = \{[ \FN] \}\] be a maximal $\calH$--invariant filtration by free factor systems such that each multi-edge extension is non-geometric.  Then there exists an element $\varphi \in \calH$ such that for each $i = 1,\ldots,k$ where $\calF_{i-1} \sqsubset \calF_{i}$ is a multi-edge extension, $\varphi$ is irreducible and non-geometric with respect to $\calF_{i-1} \sqsubset \calF_{i}$.
\end{corollary}

\begin{proof}
The proof is the same as the proof of \cite[Theorem~6.6]{CU}, as commented in Remark~\ref{rem:HMdecomp}.  The key point is that Theorems~\ref{th:HMdecomp}, \ref{th:multi edge dichotomy} and \ref{th:relativecosurface} provide for the existence of $\delta$--hyperbolic spaces corresponding to each multi-edge extension and for each an element which acts as a hyperbolic isometry.  The main theorem in~\cite{CU} shows that under these hypotheses, there is a single element in $\calH$ which acts as a hyperbolic isometry in each.  Applying Theorem~\ref{th:relativecosurface} again completes the proof.
\end{proof}


\section{Dynamics on single-edge extensions}\label{sec:single-edge}

In this section we analyze the dynamics of outer automorphisms that preserve a single-edge extension of free factor systems $\calF_{0} \sqsubset \calF_{1}$.  The main result of this section is that in the most interesting case of a handle extension, if $\varphi$ preserves the extension and acts as an atoroidal element on $F_{N-1}$, then $\varphi$ acts on the space of currents on $F_N$ with generalized north-south dynamics (Theorem~\ref{th:gns}). 


\subsection{Almost atoroidal elements}\label{subsec:almost atoroidal}

To begin, we characterize outer automorphisms preserving a single-edge extension $\calF_{0} \sqsubset \calF_{1}$ whose restriction to $\calF_{0}$ is atoroidal.

\begin{proposition}\label{prop:co rank one atoroidal} 
Suppose $\calF_{0} \sqsubset \calF_{1}$ is a single-edge extension of free factor systems that is invariant under $\varphi \in \IA_{N}(\ZZ/3)$.  If $\varphi\big|_{\calF_{0}}$ is atoroidal, then one of the following holds.
\begin{enumerate} 

\item\label{item:one edge atoroidal} The restriction $\varphi\big|_{\calF_{1}}$ is atoroidal. 

\item\label{item:one edge fixed}  There exists a nontrivial $g\in\FN$ such that  $g$, its inverse, and its iterates are the only nontrivial conjugacy classes in $\calF_1$ fixed by $\varphi\big|_{\calF_{1}}$.  Furthermore, there is some $[A] \in \calF_{0}$ such that either:
\begin{itemize}
\item $\calF_{1} = \calF_{0} \cup \{[\I{g}]\}$ (circle extension); or
\item $\calF_{1} = \bigr( \calF_{0} - \{[A]\} \bigl)\cup \{[A \ast \I{g}]\}$ (handle extension).
\end{itemize}
\end{enumerate}
\end{proposition}

\begin{proof}  
Let $f \from G\to G$ be a $\CT$ that represents $\varphi^{M}$ and realizes $\calC = (\calF_{0},\calF_{1})$, where $M$ is the constant from Theorem~\ref{th:CT exist}.  Let $H_r$ be the $\NEG$ stratum corresponding to the extension $\calF_{0}\sqsubset\calF_{1}$, i.e., $\calF_{0} = \calF(G_{r-1})$, $\calF_{1} = \calF(G_{r})$ and $H_r=\overline{G_r - G_{r-1}}$.  By Lemma~\ref{CT-properties}\eqref{CT-properties:NEG}, $H_{r}$ consists of a single edge $e$.  

If $\calF_0 \sqsubset \calF_1$ is a circle extension, then the second statement of the proposition holds.  Else, if $\calF_0 \sqsubset \calF_1$ is a barbell extension, then $\varphi\big|_{\calF_1}$ is atoroidal and so the first statement of the proposition holds.  Hence we assume that $\calF_0 \sqsubset \calF_1$ is a handle extension.  Let $[A] \in \calF_0$ correspond to the component of $G_{r-1}$ upon which $e$ is attached.

First, suppose that $e$ is a linear edge, i.e., $f(e) = e \rho$ where $\rho$ is a nontrivial closed Nielsen path in $G_{r-1}$.  Then the conjugacy class corresponding to $\rho$ is fixed by $\varphi$ and is in $\calF_{0}$, contradicting the assumption $\varphi\big|_{\calF_{0}}$ is atoroidal.  Hence this case does not occur.

Next, suppose that $e$ is a fixed edge.  If $o(e) = t(e)$, we claim that the conjugacy class $g$ that corresponds to the loop $e$ is the only fixed conjugacy class up to inversion and taking powers.  Thus the second statement of the proposition holds.  Indeed, any other conjugacy class $[h]$ in $\calF_{1}$ is represented by a cyclically reduced loop of the form $e^{a_{1}}\alpha_{1}e^{a_{2}} \ldots \alpha_{k}$ where the $\alpha_{i}$'s are reduced loops in $G_{r-1}$ based at the common vertex $o(e)=t(e)$ and the $a_i$'s are non-zero integers.  If $\varphi^{Mp}[h]=[h]$ for some $p \geq 1$, then $[f^{p}(e^{a_{1}}\alpha_{1}e^{a_{2}} \ldots \alpha_{k})] = \sigma e^{a_{1}}\alpha_{1}e^{a_{2}} \ldots \alpha_{k} \sigma^{-1}$ for some reduced edge path $\sigma$ (note, the image path is reduced except possibly at $\sigma \cdot e^{a_{1}}$ or $\alpha_{k} \cdot \sigma\inv$).  Since $f(e)=e$ and $f$ preserves $G_{r-1}$,  $f^{p}$ must permute the $\alpha_i$'s (up to homotopy rel endpoints).  Hence some power of $f$ fixes each $\alpha_{i}$ which is a contradiction as the restriction of $\varphi$ to $\calF_{0}$ is atoroidal.   
  
If $o(e) \neq t(e)$, we claim that there can be at most one fixed conjugacy class in $\calF_{1}$ up to inversion and taking powers.  Thus the second statement of the proposition holds.  Indeed, suppose $h_{1}, h_{2} \in \FN$ are not proper powers, $[h_{1}]$ and $[h_{2}]$ are in $\calF_{1}$, and are fixed by $\varphi$.  As the restriction of $\varphi$ to $\calF_{0}$ is atoroidal, we have that $[h_{1}]$ is represented by a cyclically reduced loop $e^{a_{1}}\alpha_{1}e^{a_{2}} \ldots \alpha_{k}$ where the $\alpha_{i}$'s are reduced paths in $G_{r-1}$ and each $a_{i} \in \{-1,1\}$.  Similarly, $[h_{2}]$ is represented by a cyclically reduced loop $e^{b_{1}}\beta_{1}e^{b_{2}} \ldots \beta_{\ell}$ where again the $\beta_{i}$'s are reduced paths in $G_{r-1}$ and each $b_{i} \in \{-1,1\}$.  As in the previous case of a loop, some power of $f$ fixes each $\alpha_{i}$ and $\beta_{i}$ (up to homotopy rel endpoints).  If there is some $i$ such that $a_{i} \neq a_{i+1}$, then the path $\alpha_{i}$ is closed and represents a conjugacy class in $\calF_{0}$ which is $\varphi$--periodic, contradicting the assumption that the restriction of $\varphi$ to $\calF_{0}$ is atoroidal.  Similarly for the $b_{i}$'s.  Thus, after possibly replacing $h_{1}$ or $h_{2}$ by their inverse, we have that each $a_{i}$ and $b_{i}$ equals $1$.  If there exist $i \neq j$ such that $\alpha_{i} \neq \alpha_{j}$, then the nontrivial closed loop $\alpha_{i}\alpha_{j}^{-1}$ is fixed by this power of $f$ and contained in $G_{r-1}$, again contradicting the assumption that the restriction of $\varphi$ to $\calF_{0}$ is atoroidal.  Thus the $\alpha_{i}$'s are all the same path $\alpha$ and since $h_{1}$ is not a proper power, we have that $[h_{1}]$ is represented by the cyclically reduced path $e\alpha$.  Similarly $[h_{2}]$ is represented by the cyclically reduced path $e\beta$.  Finally, if $\alpha \neq \beta$, then the nontrivial closed loop $\alpha\beta^{-1}$ is fixed by a power of $f$, again contradicting the assumption that the restriction of $\varphi$ to $\calF_{0}$ is atoroidal.  Hence $[h_{1}] = [h_{2}]$.     

Lastly, in the remaining case that $e$ is superlinear, there is no Nielsen path that crosses $e$ \cite[Fact~1.43]{HMpart1}, hence the restriction of $\varphi$ to $\calF_{1}$ is atoroidal as well.  Thus the first statement of the proposition holds. 

In all cases, we see that $\varphi$ has at most one fixed conjugacy class up to taking powers and inversion which proves the first part of the theorem.  The last assertion for the second statement follows from the fact that the path representing a possible fixed $g$ crosses the edge $e$ exactly once, see for example \cite[Corollary~3.2.2]{BFH00}. 
\end{proof}


\subsection{North-south dynamics for atoroidal elements}\label{subsec:ns-atoroidal}
 
The second author recently proved that atoroidal elements of $\Out(\FN)$ act on $\PCurr(\FN)$ with north-south dynamics in the following sense.
 
\begin{theorem}[{\cite[Theorem~1.4]{U3}}]\label{dynamicsofhyp} 
Let $\varphi\in \Out(\FN)$ be an atoroidal outer automorphism of a free group of rank $N\ge3$. There are simplices $\Delta_{+}$, $\Delta_{-}$ in $\PCurr(\FN)$ such that $\varphi$ acts on $\PCurr(\FN)$ with north-south dynamics from $\Delta_{-}$ to $\Delta_{+}$.  Specifically, given open neighborhoods $U$ of $\Delta_{+}$ and $V$ of $\Delta_{-}$ there exists $M > 0$ such that $\varphi^{n}(\PCurr(\FN) - V)\subset U$, and $\varphi^{-n}(\PCurr(\FN) - U)\subset V$ for all $n\ge M$. 
\end{theorem}

We also need the following statement regarding the behavior of the length of a current under iteration of $\varphi$.  In this statement, we assume $\varphi \in \Out(\FN)$ satisfies the hypotheses of Theorem~\ref{dynamicsofhyp} and $\Delta_{-}$ is the $\varphi$--invariant simplex in $\PCurr(\FN)$ appearing in the statement of that theorem.  

\begin{lemma}[{cf. \cite[Corollary~4.13]{KL5}}]\label{lem:dynamics in simplex}
For each $C > 0$ and neighborhood $V$ of $\Delta_{-}$ there is a constant $M > 0$ such that if $[\mu] \notin V$, then $\wght{\varphi^{n}\mu} \geq C\wght{\mu}$ for all $n \geq M$.
\end{lemma}

A similar statement appears as Lemma~\ref{lem:growth outside of nbhd}.  The proof given there directly adapts to prove this statement.

\subsection{Completely split goodness of paths and currents}\label{subsec:goodness}

To deal with single-edge extensions, we need similar statements for an element of $\Out(\FN)$ that restricts to an atoroidal element on a \emph{co-rank $1$ free factor} of $\FN$, i.e., a free factor $A < \FN$ for which there exists a nontrivial $g \in \FN$ such that $\FN = A \ast \I{g}$.  This is the purpose of this subsection and the next where we describe the necessary tools to prove Theorem~\ref{th:gns}.  The majority of the work in the next two section modifies the constructions and argument in \cite{U3} to deal with the free factor $\I{g}$.  A casual reader can review the main statements corresponding to the two above, Theorem~\ref{th:gns} and Lemma~\ref{lem:growth outside of nbhd}, and skip ahead to Section~\ref{sec:push past one edge}.

\begin{assumption}\label{stand}
Suppose $A < \FN$ is a co-rank $1$ free factor and $\varphi \in \IA_{N}(\ZZ/3) \cap \Out(\FN;A)$ is such that $\varphi\big|_{A}$ is atoroidal.  Let $\Delta_{+}$ and $\Delta_{-}$ be the inclusion to $\PCurr(\FN)$ of the $\varphi$--invariant simplices in $\PCurr(A)$ from Theorem~\ref{dynamicsofhyp} for $\varphi\big|_{A}$.  Assume $\varphi$ is not atoroidal and let $[g]$ be the fixed conjugacy class in $\FN$ given by Proposition~\ref{prop:co rank one atoroidal}\eqref{item:one edge fixed}.  Let
\begin{equation*}
\hDelta_{-}=\{[t\eta_{g}+(1-t)\mu_{-}]\mid [\mu_{-}]\in\Delta_{-}, t\in[0,1]\}
\end{equation*}
and 
\begin{equation*}
\hDelta_{+}=\{[t\eta_{g}+(1-t)\mu_{+}]\mid [\mu_{+}]\in\Delta_{+}, t\in[0,1]\}.
\end{equation*}
\end{assumption}

Throughout the rest of this section and the next, we will further assume the element $\varphi$ is represented by a CT map $f \from G \to G$ in which the fixed conjugacy class $[g]$ is represented by a loop edge $e$ in $G$ which is fixed by $f$.  The complement of the edge $e$ in $G$ is denoted $G'$.  This assumption is not a restriction (upon replacing $\varphi$ by a sufficient power to ensure some CT).  Indeed, if in the proof of Proposition \ref{prop:co rank one atoroidal} the edge $e$ is a loop edge we are done.  Otherwise, the conclusion of Proposition \ref{prop:co rank one atoroidal} says that $[g]$ is a free factor so we can take a CT map $f' \from G'\to G'$ that represents $\varphi\big|_{A}$ and let $G = G' \vee e$ where the wedge point is at an $f'$-fixed vertex and $e$ is a loop edge representing $[g]$.  There is an obvious extension to a map $f \from G \to G$ representing $\varphi \in \Out(\FN)$ that is a CT map.  Existence of a fixed vertex is guaranteed by the properties of CT's, see \cite[Definition 3.18 and Lemma 3.19]{FH}.

A decomposition of a path $\gamma$ in $G$ into subpaths $\gamma = \gamma_1\cdot\gamma_2\cdot \ldots \cdot \gamma_n$ is called a \emph{splitting} if for all $k \geq 0$ we have 
\begin{equation*}
[f^{k}(\gamma)]=[f^{k}(\gamma_1)][f^{k}(\gamma_2)]\ldots[f^{k}(\gamma_n)].
\end{equation*}
In other words, any cancellation takes place within the images of the $\gamma_{i}$'s.  We use the ``$\cdot$'' notation for splittings.  A path $\gamma$ is said to be \emph{completely split} if it has a splitting $\gamma_1 \cdot \gamma_2 \cdot \ldots \cdot \gamma_{n}$ where each $\gamma_{i}$ is either an edge in an irreducible stratum, an indivisible Nielsen path or a maximal taken connecting path in a zero stratum.  These type of subpaths are called \emph{splitting units}.  We refer reader to \cite{FH} for complete details and note that the assumption on $\varphi$ above guarantees that there are no exceptional paths.  Of importance is that if $\gamma = \gamma_1\cdot\gamma_2\cdot\ldots\cdot\gamma_n$ is a complete splitting, then $[f(\gamma)]$ also has a complete splitting where the units refine $[f(\gamma)]=[f(\gamma_1)]\cdot [f(\gamma_2)] \cdot  \ldots  \cdot [f(\gamma_n)]$~\cite[Lemma~4.6]{FH}.  We say that a splitting unit $\sigma$ is \emph{expanding} if $|[f^{k}(\sigma)]|\to\infty$ as $k\to\infty$.  Recall $|\param|$ denotes the simplicial length of a path.

We next need to introduce a notion of \emph{goodness} tailored to the setting of CT maps.  Goodness appears in several places in the literature~\cite{BFH97,Martin,Uyaiwip}.  Intuitively, the closer the goodness of a path is to $1$, the more we understand the qualitative behavior of its forward images.  In the previous settings, it is defined using legal subpaths, in the current setting, completely split subpaths are the relevant piece to keep track of. 

\begin{definition}\label{def:goodness of paths}
For an edge path $\gamma$ in $G$, a \emph{maximal splitting} is a splitting $\gamma = \beta_{0} \cdot \alpha_{1} \cdot \beta_{1} \cdot \ldots \cdot \alpha_{n} \cdot \beta_{n}$ where each $\alpha_{i}$ has a complete splitting, $\beta_{i}$ is nontrivial for $i = 1,\ldots,n-1$ and $\sum_{i=1}^{n} |\alpha_{i}|$ is maximized.  Using a maximal splitting, we define the \emph{completely split goodness} of $\gamma$ as:
\begin{equation*}
\frakg(\gamma) = \frac{1}{|\gamma|} \sum_{i=1}^{n} |\alpha_{i}|.
\end{equation*}
\end{definition}

If $\gamma$ is a cyclically reduced circuit in $G$, set $\frakg(\gamma)$ to be the maximum of $\frakg(\gamma')$ over all cyclic permutations of $\gamma$.  For any conjugacy class $h \in \FN$, let $\gamma_{h}$ be the unique cyclically reduced circuit in $G$ that represents $[h]$.  We define the \emph{completely split goodness} of a conjugacy class $[h]$ as $\frakg([h]) = \frakg(\gamma_{h})$.  It is not clear that $\frakg$ can extend in a continuous way to $\Curr(\FN)$.  What we can do is to define a continuous function $\overline{\frakg} \from \Curr(\FN) \to \RR$ that agrees with $\frakg$ on completely split circuits and provides a lower bound on $\frakg$ in general.  The first ingredient is the bounded cancellation lemma.

\begin{lemma}\cite{Coo}\label{BCL} Let $f \from G \to G$ be a graph map. There exists a constant $C_{f}$ such that for any reduced path $\gamma=\gamma_{1}\gamma_{2}$ in $G$ one has
\[
|[f(\gamma)]|\ge|[f(\gamma_1)]|+|[f(\gamma_2)]|-2C_{f}.
\]
\end{lemma}

Let $C_0$ be the maximum length of a Nielsen path or a taken connecting path in a zero stratum in $G'$.  Finiteness of $C_{0}$ follows as $\varphi\big|_{A}$ is atoroidal and zero strata are contractible.  This same $C_0$ also works for $f^{k}$ for all  $k \geq 1$.  We now replace the CT map $f$ with a suitable power, but continue to use $f$, so that for each expanding splitting unit $\sigma$, we have $\abs{[f(\sigma)]} \geq 3(2C_0+1)\abs{\sigma}$.  Let $C_f$ be the bounded cancellation constant for this new $f$ and $C = \max\{C_0+1,C_f\}$. 

\begin{proposition}\label{growthprop} 
Under the standing assumption \ref{stand}, the following hold:
\begin{enumerate}

\item If a path $\gamma$ in $G'$ is completely split and $|\gamma| \geq C_0+1$, then:
\begin{equation*}
\frac{\textit{sum of lengths of expanding splitting units}}{|\gamma|}\geq \frac{1}{2C_0+1}.
\end{equation*}

\item If a path $\gamma$ in $G'$ is completely split and $|\gamma| \geq C_{0} + 1$, then: 
\begin{equation*}
|[f(\gamma))]| \geq 3|\gamma|.
\end{equation*}

\item\label{item:good} Let $\gamma$ be any path in $G$ and suppose $\gamma_{0} \cdot \gamma_{1} \cdot \gamma_{2}$ is a subpath of $\gamma$ where each $\gamma_{i}$ has a complete splitting.  If $|\gamma_{0}|, |\gamma_{2}| \geq C$ then $\gamma$ has a splitting $\gamma = \gamma' \cdot \gamma_{1} \cdot \gamma''$.

\end{enumerate}
\end{proposition}

\begin{proof}
The proof of (1) is similar to that of~\cite[Proposition 3.9]{U3}.  Properties of CT's imply that $\gamma$ has a splitting $\gamma = \beta_{0} \cdot \alpha_{1} \cdot \beta_{1} \cdot  \ldots \cdot \alpha_{n} \cdot \beta_{n}$ where each $\alpha_{i}$ has a complete splitting into edges in EG strata (in particular into expanding splitting units) and each $\beta_{j}$ is either a Nielsen path or a taken connecting path in a zero stratum.  Since $|\gamma| \geq C_{0}$ we must have $n > 0$.  As $|\alpha_{i}| \geq 1$ for all $i$ and $|\beta_{j}| \leq C_{0}$ for all $j$ we have:
\begin{equation*}
\frac{|\gamma|}{\sum_{i=1}^{n}|\alpha_{i}|} = 1 + \frac{\sum_{j=0}^{n+1}|b_{j}|}{\sum_{i=1}^{n}|\alpha_{i}|} \leq 1 + \frac{(n+1)C_{0}}{n} \leq 2C_{0} + 1.
\end{equation*}
Therefore:
\begin{align*}
\frac{\textit{sum of lengths of expanding splitting units}}{|\gamma|} & \geq\frac{\sum_{i=1}^{n} |\alpha_{i}|}{|\gamma|} \geq \frac{1}{2C_{0} + 1}.
\end{align*}

We get (2) by noting that $|[f(\alpha_{i})]| \geq 3(2C_{0} + 1)|\alpha_{i}|$ for all $i$ and so by (1):
\begin{equation*}
|[f(\gamma)]| \geq \sum_{i=1}^{n}|[f(\alpha_{i})| \geq 3(2C_{0} + 1)\sum_{i=1}^{n}|\alpha_{i}| \geq 3|\gamma|.
\end{equation*}

For (3) we first observe that by (2), we have $|[f(\gamma_{0})]|, |[f(\gamma_{2})]| \geq 3C \geq C_{f} + C_{0} + C$.  Decompose $\gamma$ as a concatenation $\gamma = \gamma'_{0}\gamma_{0}\gamma_{1}\gamma_{2}\gamma'_{2}$.  Applying Lemma~\ref{BCL} to $\gamma' = \gamma'_{0}\gamma_{0}$ we get that at most $C_{f}$ edges of $[f(\gamma'_{0})]$ cancels with $[f(\gamma_{0})]$ and therefore, the terminal segment of length $C + C_{0}$ in $[f(\gamma_{0})]$ remains in $[f(\gamma')]$.  As $[f(\gamma_{0})]$ is completely split, we see that $[f(\gamma')] = \gamma''_{0}\hat\gamma_{0}$ where $\hat\gamma_{0} \subseteq [f(\gamma_{0})]$ is completely split and $|\hat\gamma_{0}| \geq C$.  Likewise for $\gamma'' = \gamma_{2}\gamma'_{2}$ we see that $[f(\gamma'')] = \hat\gamma_{2}\gamma''_{2}$ where $\hat\gamma_{2} \subseteq [f(\gamma_{2})]$ is completely split and $|\hat\gamma_{2}| \geq C$.  

As $\gamma_{0} \cdot \gamma_{1} \cdot \gamma_{2}$ is a splitting, we have $[f(\gamma)] = [f(\gamma')][f(\gamma_{1})][f(\gamma'')]$.

Since the path $\hat\gamma_{0} \cdot f(\gamma_{1}) \cdot \hat\gamma_{2}$ is a subpath of $[f(\gamma)]$ satisfying the same hypotheses as $\gamma_{0} \cdot \gamma_{1} \cdot \gamma_{2}$ did for $\gamma$, we can repeatedly apply this argument to get $[f^{k}(\gamma)] = [f^{k}(\gamma')][f^{k}(\gamma_{1})][f^{k}(\gamma'')]$ for all $k \geq 1$ and so $ \gamma = \gamma' \cdot \gamma_{1} \cdot \gamma''$ is a splitting.
\end{proof}

Let $\calP_{\rm cs}$ denote the set of paths in $G$ that have a complete splitting comprised of exactly $2C+1$ splitting units.  Given $\gamma \in \calP_{\rm cs}$ we have $\gamma = \sigma_{1} \cdot \sigma_{2} \cdot \ldots  \cdot \sigma_{2C+1}$ where each $\sigma_{i}$ is a splitting unit and we define $\check\gamma = \sigma_{C+1}$, i.e., the middle splitting unit.  It is possible that distinct paths $\gamma, \gamma' \in \calP_{\rm cs}$ could be nested, i.e., $\gamma' \subsetneq \gamma$.  For instance, if the first or last unit in $\gamma$ is either an indivisible Nielsen path or a taken connecting path in a zero stratum then it is possible that $\gamma$ has a completely split subpath $\gamma'$ with $2C+1$ terms where the first and/or last terms are either edges in the indivisible Nielsen path or a smaller taken connecting zero path.  For such $\check\gamma = \check\gamma'$.  We need to keep track of such behavior and so define:
\begin{equation*}
\calP_{\rm cs}^{\rm min} = \{ \gamma \in \calP_{\rm cs} \mid \not\exists \gamma' \in \calP_{\rm cs} \mbox{ where } \gamma \subsetneq \gamma' \mbox{ and } \check\gamma = \check\gamma'\}.
\end{equation*}    

We can now define a version of completely split goodness for currents.

\begin{definition}\label{def:goodness of current}
For any non-zero $\mu\in \Curr(\FN)$ define the \emph{completely split goodness} of $\mu$ by:
\begin{equation}\label{eq:csg}
\overline{\frakg}(\mu)=\frac{1}{|\mu|}\sum_{\gamma \in \calP_{\rm cs}^{\rm min}}\I{\gamma, \mu} |\check\gamma|.
\end{equation}
\end{definition}

Observe that $\overline{\frakg}$ descends to a well-defined function $\overline{\frakg}\from \PCurr(\FN) \to \RR$.  The important properties of $\overline{\frakg}$ are summarized in the following lemma.

\begin{lemma}\label{generalizedgoodnessproperties}
The map $\overline{\frakg}\from \Curr(\FN)-\{0\} \to \RR$ is continuous.  Further for any rational current $\eta_{h}$:

\begin{enumerate}
\item $\overline{\frakg}(\eta_{h})=1$ if $\eta_{h}$ is represented by a completely split circuit; and 

\item $\frakg(\gamma_{h}) \geq \overline{\frakg}(\eta_{h})$ where $\gamma_{h}$ is the unique reduced circuit in $G$ that represents $[h]$.
\end{enumerate}
\end{lemma}

\begin{proof} The continuity is clear as it is defined using linear combination of continuous functions (Lemma~\ref{lem:topology}). 

For the first assertion, suppose $h$ is represented by a completely split cyclically reduced circuit $\gamma = \sigma_{1} \cdot \sigma_{2} \cdot \ldots \cdot \sigma_{n}$.  For each $i$, the path:
\begin{equation*}
\gamma_{i} = \sigma_{i-C} \cdot \, \cdots \, \cdot \sigma_{i-1} \cdot \sigma_{i} \cdot \sigma_{i + 1} \cdot \, \cdots \, \cdot \sigma_{i + C}
\end{equation*}
where the indices are taken modulo $n$ is in $\calP_{\rm cs}$ and has $\check\gamma_{i} = \sigma_{i}$.  Thus each splitting unit $\sigma_{i}$ in $\gamma$ is the middle term of completely split edge path of length $2C+1$. The minimal such path contributes to the right-hand side of \eqref{eq:csg} the number of edges of $\sigma_{i}$. 

The second assertion follows from Proposition \ref{growthprop}\eqref{item:good}.
\end{proof}


\subsection{Incorporating north-south dynamics from lower stratum}\label{subsec:goodness grows}

We need to work with the inverse outer automorphism $\varphi^{-1}$ as well.  We will denote the CT map for $\varphi$ by $f_{+} \from G_{+} \to G_{+}$.  As in Section~\ref{subsec:goodness}, we assume that there is an edge $e_{+}$ in $G_{+}$ representing the fixed conjugacy class $[g]$ and we will denote the complement of $e_{+}$ in $G_{+}$ by $G'_{+}$.  The corresponding completely split goodness function is denoted by $\overline{\frakg}_{+}$.  For $\varphi^{-1}$, we denote the corresponding objects by $f_{-} \from G_{-} \to G_{-}$, $e_{-}$, $G'_{-}$ and $\overline{\frakg}_{-}$.  Let us denote the total length of subpaths of $\gamma$ that lie in $G'_{+}$ by $|\gamma|'$, and by abuse of notation we denote the corresponding length functions on $G_{-}$ and $G'_{-}$ with $|\param|$ and $|\param|'$ as well, their use will be clear from context.  

Notice that any path $\gamma$ in $G_{+}$ has a splitting $\gamma = \alpha_{0} \cdot e_{+}^{k_{1}} \cdot \alpha_{1} \cdot \ldots \cdot e_{+}^{k_{m}} \cdot \alpha_{m}$ where each $\alpha_{i}$ is a closed path in $G'_{+}$ which is nontrivial for $i = 1,\ldots, m-1$ and each $k_{i}$ is a nonzero integer.  This follows as $f_{+}(e_{+}) = e_{+}$ and $f_{+}(G') \subseteq G'$.  If $\gamma$ is not a power of $e_{+}$ we define:
\begin{equation*}
\frakg'_{+}(\gamma) =  \frac{\sum_{i=0}^{m}|\alpha_{i}|\frakg_{+}(\alpha_{i})}{\sum_{i=0}^{m}|\alpha_{i}|}.
\end{equation*}
In other words, we are measuring the proportion of $\gamma$ in $G'$ that is completely split.  There is a similar discussion for paths in $G_{-}$ and we define $\frakg'_{-}$ analogously.

Given $h \in \FN$, we let $\gamma^{+}_{h}$ and $\gamma^{-}_{h}$ respectively denote the unique cyclically reduced circuits in $G_{+}$ and $G_{-}$ respectively that represent $[h]$.  The following proposition summarizes the key properties of $\frakg'_{+}$ and how it will be used to detect how close a current is to the attracting simplices.  

\begin{proposition}\label{prop:goodness properties}
Under the standing assumption \ref{stand}, the following hold for all $h \in \FN$ that is not conjugate to a power of $g$.
\begin{enumerate}
\item For any open neighborhood $U_{+}$ of $\Delta_{+}$ there exists a $0 < \delta < 1$  and $M > 0$ such that $\varphi^{n}[\eta_{h}] \in U_{+}$ for all $n \geq M$ if:
\begin{equation*}
\frakg'_{+}(\gamma_{h}^{+})\frac{|\gamma_{h}^{+}|'}{|\gamma_{h}^{+}|} > \delta.
\end{equation*}\label{GP:stronger}

\item For any $\epsilon > 0$ and $L \geq 0$ there exists a $0 < \delta < 1$ and $M > 0 $ such that for each $n \geq M$ there is a $[\mu] \in \Delta_{+}$ with:
\begin{equation*}
\abs{\frac{\I{\alpha,[f_{+}^{n}(\gamma_{h}^{+})]}}{|[f_{+}^{n}(\gamma_{h}^{+})]|'} - \frac{\I{\alpha,\mu}}{|\mu|}} < \epsilon
\end{equation*}
for every reduced path $\alpha$ in $G'_{+}$ of length at most $L$ if $\frakg'_{+}(\gamma_{h}^{+}) > \delta$.\label{GP:weaker}
\end{enumerate}
\end{proposition}

\begin{proof}
Both of these statements can be proved using arguments almost identical to  \cite[Lemma~6.1]{LU2} (see also \cite[Lemma~3.17]{U3}).  

For (1), the lower bound on this ratio implies that most of the length of $\gamma_{h}^{+}$ comes from completely split subpaths in $G_{+}'$.  The argument in~\cite[Lemma~6.1]{LU2} converts this notion to having powers that are close to currents in $\Delta_{+}$.

For (2), the lower bound on $\frakg'_{+}$ implies that most of the length of $\gamma_{h}^{+}$ contained in $G_{+}'$ comes from completely split subpaths in $G_{+}'$.  The argument in ~\cite[Lemma~6.1]{LU2} converts this notion to having powers that almost agree with currents in $\Delta_{+}$ on most subpaths of $G_{+}'$.     
\end{proof}

There of course are analogous statements for $\frakg'_{-}$.

\begin{lemma}\label{lem:backandforthlowergoodness}
Under the standing assumption \ref{stand}, given $0 < \delta < 1$ and $K \geq 0$, 
there exists an $M > 0$ such that for all $h \in \FN$ that is not conjugate to a power of $g$ either:
\begin{align*}
\frakg'_{+}([f_{+}^{n}(\gamma_{h}^{+})]) > \delta & \mbox{ and } |[f_{+}^{n}(\gamma_{h}^{+})]|' \geq K|\gamma_{h}^{+}|'; \mbox{ or}
\\
\frakg'_{-}([f_{-}^{n}(\gamma_{h}^{-})]) > \delta & \mbox{ and } |[f_{-}^{n}(\gamma_{h}^{-})]|' \geq K|\gamma_{h}^{-}|'
\end{align*}
for all $n \geq M$.
\end{lemma}

\begin{proof} 
Since the restrictions of $f_{+}$ to $G'_{+}$ and $f_{-}$ to $G'_{-}$ are atoroidal, the result essentially follows from \cite{U3}.  Indeed, writing:
\begin{align*}
\gamma_{h}^{+} & = \alpha_{0} \cdot e_{+}^{k_{1}} \cdot \alpha_{1} \cdot  \ldots  \cdot e_{+}^{k_{m}} \cdot \alpha_{m} \\
\gamma_{h}^{-} & = \beta_{0} \cdot e_{-}^{k_{1}} \cdot \beta_{1} \cdot  \ldots  \cdot e_{-}^{k_{m}} \cdot \beta_{m}
\end{align*}
we have that \cite[Lemma~3.19]{U3} provides the existence of an $M_{0}$ such that for each pair $\{\alpha_{i},\beta_{i}\}$ we have that one of $\frakg_{+}([f_{+}^{M_{0}}(\alpha_{i})])$ or $\frakg_{-}([f_{-}^{M_{0}}(\beta_{i})])$ is at least $\frac{1}{2}$.  Let $J \subseteq \{0,1,\ldots,m\}$ be the subset where the first alternative occurs.  Let $L \geq 1$ be such that $\frac{1}{L}|[f_+^{M_{0}}(\alpha_{i})]| \leq |[f_-^{M_{0}}(\beta_{i})]| \leq L|[f_+^{M_{0}}(\alpha_{i})]|$ for each $i$.

Suppose that $\sum_{i \in J} |[f_+^{M_{0}}(\alpha_{i})]| \geq \frac{1}{2}\sum_{i=0}^{m}|[f_+^{M_{0}}(\alpha_{i})]|$.  Then:
\begin{align*}
\frakg'_{+}([f_{+}^{M_{0}}(\gamma_{h}^{+})]) & =  \frac{\sum_{i=0}^{m}|[f_{+}^{M_{0}}(\alpha_{i})]|\frakg_{+}([f^{M_{0}}_{+}(\alpha_{i})])}{\sum_{i=0}^{m}|[f_{+}^{M_{0}}(\alpha_{i})]|} \\
& \geq \frac{1}{2} \frac{\sum_{i \in J}|[f_{+}^{M_{0}}(\alpha_{i})]|\frakg_{+}([f^{M_{0}}_{+}(\alpha_{i})])}{\sum_{i \in J}|[f_{+}^{M_{0}}(\alpha_{i})]|} \\
& \geq \frac{1}{4}. 
\end{align*}

Otherwise we have $\sum_{i \notin J}|[f_+^{M_{0}}(\alpha_{i})]|\geq \frac{1}{2}\sum_{i=0}^{m}|[f_+^{M_{0}}(\alpha_{i})]|$ and so:
\begin{align*}
\sum_{i \notin J} |[f_-^{M_{0}}(\beta_{i})]| \geq & \frac{1}{L}\sum_{i \notin J}|[f_+^{M_{0}}(\alpha_{i})]| \\ 
\geq & \frac{1}{2L}\sum_{i=0}^{m}|[f_+^{M_{0}}(\alpha_{i})]| \geq \frac{1}{2L^{2}}\sum_{i=0}^{m}|[f_-^{M_{0}}(\beta_{i})]|.
\end{align*}
A similar calculation in this case shows that $\frakg'_{-}([f_{-}^{M_{0}}(\gamma_{h}^{-})]) \geq \frac{1}{4L^{2}}$ in this case.

Next, the proof of \cite[Lemma~3.16]{U3} provides the existence of an $M_{1}$ such that if $\frakg'_{\pm}(\gamma) \geq \frac{1}{4L^{2}}$ then $\frakg'_{\pm}([f_{\pm}^{n}(\gamma)]) > \delta$ for $n \geq M_{1}$.  Finally, the proof of \cite[Lemma~3.14]{U3} provides the existence of an $M_{2}$ such that if $\frakg'_{\pm}(\gamma) > 0$, then $\frakg'_{\pm}([f^{n}_{\pm}(\gamma)]) > \frakg'_{\pm}(\gamma)$ for all $n \geq M_{2}$.  Hence for $M = M_{0}M_{1} + M_{2}$ we have that the first conclusion of the alternative holds.

The second conclusion of the alternative follows from the proof of~\cite[Lemma~3.16]{U3} as well.  Indeed, in this lemma, it is shown that  for each $0 < \delta' < 1$ there is a $\lambda > 0$ such that if $\frakg_{\pm}(\gamma) \geq \delta'$ where $\gamma$ is a path in $G_{\pm}'$ then $|f_{\pm}^{n}(\gamma)| \geq 2^{n}\lambda|\gamma|$.  The argument now proceeds like above using a possibly larger $M$.
\end{proof}

Combining the two previous statements, we can show north-south dynamics on $\PCurr(\FN)$ outside of a neighborhood of the fixed point $[\eta_{g}]$. 

\begin{proposition}\label{prop:backandforth}
Under the standing assumption \ref{stand}, given open neighborhoods $U_{\pm}$ of $\Delta_{\pm}$ and $W$ of $[\eta_{g}]$ there is an $M > 0$ such that for any rational current $[\eta_{h}]\in \PCurr(\FN) - W$, either $\varphi^{n}[\eta_{h}]\in U_{+}$ or $\varphi^{-n}[\eta_{h}]\in U_-$ for all $n\ge M$.\end{proposition}

\begin{proof}
To begin, we observe that $\frac{\langle e_{\pm}, \mu\rangle}{|\mu|}=1$ if and only if $[\mu]=[\eta_{g}]$.  Hence by continuity of $\I{e_{+},\param}$ and compactness of $\PCurr(\FN)$, there is an $0 < s < 1$ such that $\frac{\langle e_{\pm}, \mu\rangle}{|\mu|} \leq 1 - s$ for $[\mu] \notin W$.  

Let $0 < \delta_{0} < 1$ and $M_{0}$ be the maximum of constants from Proposition~\ref{prop:goodness properties}\eqref{GP:stronger} using both $U_{+}$ and $U_{-}$.  Set $\delta = \sqrt{\delta_{0}}$ and $K > 1$ large enough so that $\frac{K}{K + 1/s} > \sqrt{\delta_{0}}$.  Finally, let $M_{1}$ be the constant from Lemma~\ref{lem:backandforthlowergoodness} using these constants.  Suppose that $[\eta_{h}] \notin W$ and without loss of generality assume that the first alternative of Lemma~\ref{lem:backandforthlowergoodness} holds for $h$.  As $|\gamma_{h}^{+}| = |\gamma_{h}^{+}|' + \I{e_{+},\gamma_{h}^{+}}$ we get $|\gamma_{h}^{+}|'/|\gamma_{h}^{+}| \geq s$ and so $\frac{\I{e_{+},\gamma_{h}^{+}}}{|\gamma_{h}^{+}|'} \leq \frac{1-s}{s} < \frac{1}{s}$. 

Therefore we find:
\begin{align*}
\dfrac{|[f^{M_{1}}_{+}(\gamma_{h}^{+})]|'}{|[f^{M_{1}}_{+}(\gamma_{h}^{+})]|} & = 
\dfrac{|[f^{M_{1}}_{+}(\gamma_{h}^{+})]|'}{|[f^{M_{1}}_{+}(\gamma_{h}^{+})]|' + \I{e_{+},\gamma_{h}^{+}}} = \dfrac{1}{1 + \dfrac{\I{e_{+},\gamma_{h}^{+}}}{|[f^{M_{1}}_{+}(\gamma_{h}^{+})]|'}} \\
& \geq \dfrac{1}{1 + \dfrac{\I{e_{+},\gamma_{h}^{+}}}{K|\gamma_{h}^{+}|'}}
\geq \dfrac{1}{1 + \dfrac{1}{Ks}} = \frac{K}{K + 1/s} > \sqrt{\delta_{0}}.
\end{align*} 
And thus:
\begin{equation*}
\dfrac{\frakg'_{+}([f_{+}^{M_{1}}(\gamma_{h}^{+})])|[f^{M_{1}}_{+}(\gamma_{h}^{+})]|'}{|[f^{M_{1}}_{+}(\gamma_{h}^{+})]|} > \delta \sqrt{\delta_{0}} = \delta_{0}.
\end{equation*}
Hence by Proposition~\ref{prop:goodness properties}\eqref{GP:stronger} we have $\varphi^{n}[\eta_{h}] \in U_{+}$ for $n \geq M = M_{0} + M_{1}$.
\end{proof}

In order to promote Proposition~\ref{prop:backandforth} to generalized north-south dynamics everywhere, we need to know that there are contracting neighborhoods.  This is content of the next two lemmas and where we need the notion of completely split goodness for currents and Lemma~\ref{generalizedgoodnessproperties}.   We have one lemma dealing with neighborhoods of $\Delta_{\pm}$ and one lemma for neighborhoods of $\hDelta_{\pm}$.

\begin{lemma}\label{lem:nbhds}
Under the standing assumption \ref{stand}, given open neighborhoods $U_{\pm}$ of $\Delta_{\pm}$ there are open neighborhoods $U_{\pm}' \subseteq U_{\pm}$ of $\Delta_{\pm}$ and such that $\varphi^{\pm 1}(U_{\pm}') \subseteq U_{\pm}'$.  
\end{lemma}

\begin{proof}
We first observe that for any point in $[\mu]\in \Delta_{+}$, the completely split goodness $\overline{\frakg}_{+}([\mu]) = 1$. This is because any such point is a linear combination of extremal points and extremal points are defined using limits of edges \cite[Proposition 3.3 and Definition 3.5]{U3}, and as $[f^{n}(e)]$ is completely split for all $n\ge1$.  Likewise $\overline{\frakg}_{-}([\mu]) = 1$ for any $[\mu] \in \Delta_{-}$.

Using these observations the conclusion of the lemma follows from the proofs of Lemma~\ref{lem:backandforthlowergoodness} and Proposition~\ref{prop:backandforth}.  To begin, given a neighborhood $U_{+}$ of $\Delta_{+}$ pick a neighborhood $U_{+}^{0} \subset U_{+}$ such that for all $[\mu]\in U_{+}^{0}$ we have $\overline{\frakg}(\mu) > \delta$ and $\frac{\langle e_{+}, \mu\rangle}{|\mu|} < s$ for some $ \delta > s >0$.  Let $0 < \delta_{0} < 1$ and $M_{0}$ be the constants from Proposition~\ref{prop:goodness properties}\eqref{GP:stronger} for $U_{+}^{0}$.  

Given $[\eta_{h}] \in U^{0}_{+}$ we find using Lemma~\ref{generalizedgoodnessproperties}:
\begin{align*}
\frakg'_{+}(\gamma_{h}^{+}) & \geq \frakg'_{+}(\gamma_{h}^{+})\frac{|\gamma_{h}^{+}|'}{|\gamma_{h}^{+}|} = \frakg_{+}(\gamma_{h}^{+}) - \frac{\I{e_{+},\gamma_{h}^{+}}}{|\gamma^{+}_{h}|} \\
& \geq \overline{\frakg}_{+}(\eta_{h}) - \frac{\I{e_{+},\eta_{h}}}{|\eta_{h}|} > \delta - s.
\end{align*}
As mentioned in the proof of Lemma~\ref{lem:backandforthlowergoodness}, there is now an $M_{1}$ such that $\frakg'_{+}([f_{+}^{n}(\gamma_{h}^{+})]) > \sqrt{\delta_{0}}$ for all $n \geq M_{1}$.  Combining now with the proof of Proposition~\ref{prop:backandforth},
for a slightly larger $M_{1}$, we have that $\frac{|[f_{+}^{n}(\gamma_{h}^{+})]|'}{|[f_{+}^{n}(\gamma_{h}^{+})]|} > \sqrt{\delta_{0}}$ as well for $n \geq M_{1}$.  By choice of $\delta_{0}$, this shows $\varphi^{M}[\eta_{h}] \in U^{0}_{+}$ for $M = M_{0} + M_{1}$ and for any rational current $[\eta_{h}] \in U_{+}^{0}$.  As rational currents are dense, we get $\varphi^{M}(U_{+}^{0}) \subseteq U_{+}^{0}$.  

Now set:
\begin{equation*}
U_{+}' = U_{+}^{0} \cap \varphi(U_{+}^{0}) \cap \cdots \cap \varphi^{M-1}(U_{+}^{0}).
\end{equation*} 
As $\varphi(\Delta_{+}) = \Delta_{+}$, $U_{+}'$ is a neighborhood of $\Delta$.  Clearly $U_{+}' \subseteq U_{+}^{0} \subseteq U_{+}$ and $\varphi(U_{+}') \subseteq U_{+}'$ by construction.

A symmetric argument works for a neighborhood of $\Delta_{-}$.
\end{proof}

\begin{lemma}\label{lem:hnbhds} 
Under the standing assumption \ref{stand}, given open neighborhoods $\hV_{\pm}$ of $\hDelta_{\pm}$ there are open neighborhoods $\hV_{\pm}' \subseteq \hV_{\pm}$ of $\hDelta_{\pm}$ such that $\varphi^{\pm 1}(\hV_{\pm}' )\subseteq \hV_{\pm}'$.  
\end{lemma}

\begin{proof}
Given $[\mu] \in \PCurr(\FN)$, a collection of reduced edge paths $\calP$ in some marked graph $G$ and an $\epsilon > 0$ determines an open neighborhood of $[\mu]$ in $\PCurr(\FN)$:
\begin{equation*}
N_{G}([\mu],\calP,\epsilon) = \left\{ [\nu] \in \PCurr(\FN) \mid \abs{\frac{\I{\gamma,\nu}}{\wght{\nu}} - \frac{\I{\gamma,\mu}}{\wght{\mu}}} < \epsilon, \, \forall \gamma \in \calP \right\}.
\end{equation*} 
For a subset $X \subseteq \PCurr(\FN)$, we define $N_{G}(X,\calP,\epsilon)$ as the union of $N_{G}([\mu],\calP,\epsilon)$ over all $[\mu] \in X$.  

By $\calP_{+}(L)$ we denote the set of all reduced edge paths contained in $G'_{+}$ with length at most $L$.  We set $\widehat\calP_{+}(L) = \calP_{+}(L) \cup \{e_{+}\}$.  We have
\begin{equation*}
\underset{L \to \infty, \, \epsilon \to 0}{\bigcap} N_{G_{+}}(\hDelta_{+},\widehat\calP_{+}(L), \epsilon) = \hDelta_{+}.
\end{equation*}
This follows as for any $[\mu] \in \Delta_{+}$, $\I{\gamma,\mu} = 0$ for any reduced edge path not contained in $G'_{+}$ and as $[\mu]=[\eta_{g}]$ if and only if $\I{e_{+}, \mu} = \wght{\mu}$.  There is a similar statement for $\hDelta_{-}$.  

Let $L$ and $\epsilon$ be such that $N_{G_{+}}(\hDelta_{+},\widehat\calP_{+}(L), \epsilon) \subseteq \hV_{+}$.  Let $\delta_{0}$ and $M_{0}$ be the constants from Proposition~\ref{prop:goodness properties}\eqref{GP:weaker} using this $L$ and $\epsilon$.  Set $\hV_{+}' = N_{G_{+}}(\hDelta_{+},\widehat\calP_{+}(L), \epsilon)$ and let $0 < \delta' < 1$ be such that $\overline{\frakg}(\mu) > \delta'$ for $[\mu] \in \hV'_{+}$.  By replacing $\delta_{0}$ with a smaller positive number and $M_{0}$ with a larger constant, we can assume that $\delta_{0}$ and $M_{0}$ also satisfy the conclusion of Proposition~\ref{prop:goodness properties}\eqref{GP:stronger} for the neighborhood $\hV'_{+}$ as well.

We will now show that there is a constant $M$ such that for any rational current $[\eta_{h}] \in \hV_{+}'$ we have $\varphi^{M}[\eta_{h}] \in \hV_{+}'$.  Arguing as in Lemma~\ref{lem:nbhds} the present lemma follows.  There are two cases: $\gamma_{h}^{+}$ has a definite fraction in $G'_{+}$; or not, i.e., $[\eta_{h}]$ is close to $[\eta_{g}]$.

The first case is similar to Lemma~\ref{lem:nbhds}.  Fix an $0 < s < \delta'$.  If $[\mu] \in \hV'_{+}$ and $\frac{\I{e_{+},\mu}}{|\mu|} < s$, then arguing as in Lemma~\ref{lem:nbhds} we have $\frakg'_{+}(\gamma_{h}^{+}) > \delta' - s$ and so there is an $M_{1}$ such that $\frakg'_{+}([f
^{n}(\gamma_{h}^{+})])\frac{|[f_{+}^{n}(\gamma_{h}^{+})]|'}{|[f^{n}_{+}(\gamma_{h}^{+})]|} > \delta_{0}$ and so $\varphi^{n}[\eta_{h}] \in \hV'_{+}$ for all $n \geq M_{0} + M_{1}$.

Thus for the second case we assume that $[\eta_{h}] \in \hV'_{+}$ and $\frac{\I{e_{+},\gamma_{h}^{+}}}{|\gamma_{h}^{+}|} \geq s$.  If $h$ is a power of a conjugate of $g$, then $\varphi([\eta_{h}]) = [\eta_{h}] \in \hV_{+}'$.  Therefore we can assume that $h$ is not a power of a conjugate of $g$.  Hence the path $\gamma^+_h$ intersects $G'_+$ nontrivially and so $|[f^{n}_{+}(\gamma_{h}^{+})]|' \geq 1$ for all $n \geq 0$.

Next we observe that given $\delta>0$ and $R>1$, there is a constant $M_{2}>1$ such that for any reduced path $\alpha$ in $G'_{+}$ which is not a Nielsen path, either $\frakg'_{+}([f_{+}^{M_2}(\alpha)])>\delta$ or $|\alpha|' > R |[f_{+}^{M_2}(\alpha)]|'$. This is the analog of \cite[Proposition~4.18]{LU2}. The idea is that any long enough reduced path $\alpha$ can be subdivided into subpaths of length at most $10C$, and we can find an exponent $M_1$ such that for any reduced edge path $\gamma$ in $G'_{+}$ with $|\gamma|<10C$, the path $[f_{+}^{M_2}(\gamma)]$ is completely split.  This tells that either $[f_{+}^{M_1}(\alpha)]$ has a definite completely split goodness, or the length $|[f_{+}^{M_1}(\alpha)]|$ decreases by a definite amount.  Hence an argument similar to the one in Lemma~\ref{lem:backandforthlowergoodness} tells that the following holds after replacing $M_{1}$ with a possibly larger constant:

For all $h \in \FN$ not conjugate to $g$, we have either: 

\begin{enumerate} 

\item $\frakg'_{+}([f_{+}^{M_1}(\gamma^{+}_h)])>\delta_{0}$; or \label{highreversegood}
\item $ |f^{M_1}_{+}(\gamma^{+}_h)|' < \dfrac{1}{R}|\gamma_h^{+}|'$ \label{lengthdecreases}

\end{enumerate}
where $\frac{1}{1 + Rs} < \epsilon$ and $\frac{R}{R + 1/s} > 1-\epsilon$.  Set $M = M_{0} + M_{1}$.


First assume that \eqref{highreversegood} holds for $h$.  Set $t = \I{e_{+},[f_{+}^{M}(\gamma_{h}^{+})]}/|[f_{+}^{M}(\gamma_{h}^{+})]|$.   As $h$ is not a power of a conjugate of $g$ we have that $0 \leq t < 1$.  As $\frakg'_{+}([f_{+}^{M_{1}}(\gamma_{h}^{+})]) > \delta_{0}$, there is a current $[\mu] \in \Delta_{+}$ satisfying the inequality in Proposition~\ref{prop:goodness properties}\eqref{GP:weaker} for $f^{M}_{+}(\gamma_{h}^{+})$.  We normalize $\mu$ so that $|\mu| = 1$.  With our normalization, we have that $|t\eta_{g} + (1-t)\mu| = 1$ as well.    We claim that $\varphi^{M}[\eta_{h}] \in N_{G_{+}}([t\eta_{g} + (1-t)\mu],\widehat\calP_{+}(L),\epsilon) \subseteq \hV'_{+}$.  

For a path $\alpha \in \calP_{+}(L)$ we have $\I{\alpha,\eta_{g}} = 0$, $|[f^{M}_{+}(\gamma_{h}^{+})]|' = |[f^{M}_{+}(\gamma_{h}^{+})]|(1-t)$ and so:
\begin{align*}
\bigg|\frac{\I{\alpha,[f_{+}^{M}(\gamma_{h}^{+})]}}{|[f_{+}^{M}(\gamma_{h}^{+})]|} - \I{\alpha,t\eta_{g} + \, & (1-t)\mu}\bigg|& \\&= \abs{\frac{\I{\alpha,[f_{+}^{M}(\gamma_{h}^{+})]}(1-t)}{|[f_{+}^{M}(\gamma_{h}^{+})]|(1-t)} - (1-t)\I{\alpha,\mu}} \\
& = \abs{\frac{\I{\alpha,[f_{+}^{M}(\gamma_{h}^{+})]}}{|[f_{+}^{M}(\gamma_{h}^{+})]|'} - \I{\alpha,\mu}}(1-t) \\
& < \epsilon(1-t) \leq \epsilon.
\end{align*}

Also as $\I{e_{+},\mu} = 0$ and $\I{e_{+},\eta_{g}} = 1$ we find:
\begin{align*}
\abs{\frac{\I{e_{+},[f_{+}^{M}(\gamma_{h}^{+})]}}{|[f_{+}^{M}(\gamma_{h}^{+})]|} - \I{e_{+},t\eta_{g} + (1-t)\mu}} &= \abs{t - t\I{e_{+},\eta_{g}}} \\
& = \abs{t - t} = 0.
\end{align*}
This shows $\varphi^{M}[\eta_{h}] \in N_{G_{+}}([t\eta_{g} + (1-t)\mu],\widehat\calP_{+}(L),\epsilon)$ as claimed.

On the other hand if \eqref{highreversegood} fails then \eqref{lengthdecreases} holds for $\gamma_{h}^{+}$ and so $|[f^{M}_{+}(\gamma_{h}^{+})]|' \leq \frac{1}{R}|\gamma_{h}^{+}|'$.  We claim that $\varphi^{M}[\eta_{h}]\in N_{G_{+}}([\eta_{g}],\widehat\calP_{+}(L), \epsilon)$.  Notice that we have $\I{e_{+},[f^{M}_{+}(\gamma_{h}^{+})]} = \I{e_{+},\gamma_{h}^{+}}$ and $\frac{\I{e_{+},\gamma_{h}^{+}}}{|\gamma_{h}^{+}|'} \geq \frac{\I{e_{+},\gamma_{h}^{+}}}{|\gamma_{h}^{+}|} \geq s$.  

For a path $\alpha \in \calP_{+}(L)$ we have $\I{\alpha,[f^{M}_{+}(\gamma_{h}^{+})]} \leq |[f^{M}_{+}(\gamma_{h}^{+})]|'$ and so:
\begin{align*}
0 < \frac{\I{\alpha,[f_{+}^{M}(\gamma_{h}^{+})]}}{|[f_{+}^{M}(\gamma_{h}^{+})]|} &\leq \frac{|[f_{+}^{M}(\gamma_{h}^{+})]|'}{|[f_{+}^{M}(\gamma_{h}^{+})]|'+\I{e_{+},[f_{+}^{M}(\gamma_{h}^{+})]}} \\
& = \frac{1}{1 + \frac{\I{e_{+},\gamma_{h}^{+}}}{|[f_{+}^{M}(\gamma_{h}^{+})]|'}} \leq \frac{1}{1 + \frac{R\I{e_{+},\gamma_{h}^{+}}}{|\gamma_{h}^{+}|'}} \\
&  \leq \frac{1}{1 + Rs} < \epsilon.
\end{align*}
Therefore as $\I{\alpha,\eta_{g}} = 0$ we have:
\begin{equation*}
\abs{\frac{\I{\alpha,[f^{M}_{+}(\gamma_{h}^{+})]}}{|[f^{M}_{+}(\gamma_{h}^{+})]|} - \I{\alpha,\eta_{g}}} < \epsilon.
\end{equation*}

Additionally, we have:
\begin{align*}
1 > \frac{\I{e_{+},[f^{M}_{+}(\gamma_{h}^{+})]}}{|[f^{M}_{+}(\gamma_{h}^{+})]|} &= 
\frac{\I{e_{+},\gamma_{h}^{+}}}{|[f^{M}_{+}(\gamma_{h}^{+})]|} = \frac{\I{e_{+},\gamma_{h}^{+}}}{|[f^{M}_{+}(\gamma_{h}^{+})]|' + \I{e_{+},\gamma_{h}^{+}}} \\
& \geq \frac{\I{e_{+},\gamma_{h}^{+}}}{\frac{1}{R}|\gamma_{h}^{+}|' + \I{e_{+},\gamma_{h}^{+}}} = \frac{R\I{e_{+},\gamma_{h}^{+}}}{|\gamma_{h}^{+}|' + R\I{e_{+},\gamma_{h}^{+}}} \\
& = \frac{R}{R + \frac{|\gamma_{h}^{+}|'}{\I{e_{+},\gamma_{h}^{+}}}} \geq \frac{R}{R + 1/s} > 1 - \epsilon.
\end{align*}
Therefore as $\I{e_{+},\eta_{g}} = 1$ we have:
\begin{equation*}
\abs{\frac{\I{e_{+},[f^{M}_{+}(\gamma_{h}^{+})]}}{|[f^{M}_{+}(\gamma_{h}^{+})]|} - \I{e_{+},\eta_{g}}} < \epsilon.
\end{equation*}
This shows $\varphi^{M}[\eta_{h}] \in N_{G_{+}}([\eta_{g}],\widehat\calP_{+}(L),\epsilon)$ as claimed.
\end{proof}


\subsection{Generalized north-south dynamics for almost atoroidal elements}\label{subsec:gns}

Using the material from the previous two sections, we can now prove the main technical result needed for Theorem~\ref{th:alternative}.

\begin{theorem}\label{th:gns}
Suppose $A < \FN$ is a co-rank $1$ free factor and $\varphi \in \IA_{N}(\ZZ/3) \cap \Out(\FN;A)$ is such that $\varphi\big|_{A}$ is atoroidal.  Let $\Delta_{+}$ and $\Delta_{-}$ be the inclusion to $\PCurr(\FN)$ of the $\varphi$--invariant simplices in $\PCurr(A)$ from Theorem~\ref{dynamicsofhyp} for $\varphi\big|_{A}$.  Assume $\varphi$ is not atoroidal and let $[g]$ be the fixed conjugacy class in $\FN$ given by Proposition~\ref{prop:co rank one atoroidal}\eqref{item:one edge fixed}.  Then $\varphi$ acts on $\PCurr(\FN)$ with generalized north-south dynamics.  Specifically, for the two invariant sets 
\begin{equation*}
\hDelta_{-}=\{[t\eta_{g}+(1-t)\mu_{-}]\mid [\mu_{-}]\in\Delta_{-}, t\in[0,1]\}
\end{equation*}
and 
\begin{equation*}
\hDelta_{+}=\{[t\eta_{g}+(1-t)\mu_{+}]\mid [\mu_{+}]\in\Delta_{+}, t\in[0,1]\},
\end{equation*}
given any open neighborhood $U_{\pm}$ of $\Delta_{\pm}$ in $\PCurr(\FN)$ and open neighborhood $\hV_{\pm}$ of $\hDelta_{\pm}$ in $\PCurr(\FN)$, there is an $M > 0$ such that $\varphi^{\pm n}(\PCurr(\FN) -  \hV_{\mp})\subset U_{\pm}$ for all $n\ge M$. 
\end{theorem}

See Figure~\ref{fig:gns set-up} for a schematic of the sets mentioned in Theorem~\ref{th:gns}.

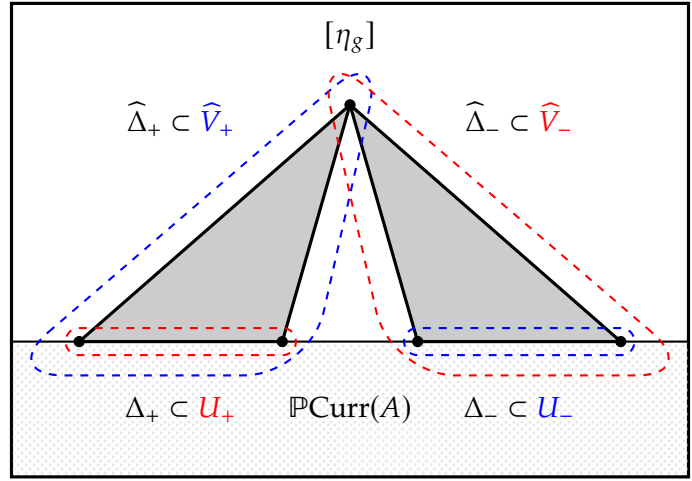
\begin{figure}[h!]	
\centering
\begin{tikzpicture}[every node/.style={inner sep=0pt},scale=0.9]
\draw[very thick] (-5,4) rectangle (5,-3);
\draw[thick,pattern=crosshatch dots,pattern color=black!10!white] (-5,-1) rectangle (5,-3);
\filldraw[black!20!white] (-4,-1) -- (0,2.5) -- (-1,-1) -- cycle;
\filldraw[black!20!white] (4,-1) -- (0,2.5) -- (1,-1) -- cycle;
\draw[very thick] (-4,-1) -- (-1,-1) node[pos=0] (a1) {} node[pos=1] (a2) {};
\draw[very thick] (4,-1) -- (1,-1) node[pos=0] (b1) {} node[pos=1] (b2) {};
\node at (0,2.5) (c) {};
\draw[very thick] (a1) -- (c) -- (a2);
\draw[very thick] (b1) -- (c) -- (b2);
\draw[thick,dashed,blue,rounded corners=25pt] (-5.2,-1.5) -- (0.5,3.4) -- (-0.6,-1.5) -- cycle;
\draw[thick,dashed,red,rounded corners=25pt] (5.2,-1.5) -- (-0.5,3.4) -- (0.6,-1.5) -- cycle;
\draw[thick,dashed,red,rounded corners=5pt] (-4.2,-0.8) -- (-0.8,-0.8) --(-0.8,-1.2) -- (-4.2,-1.2) -- cycle;
\draw[thick,dashed,blue,rounded corners=5pt] (4.2,-0.8) -- (0.8,-0.8) --(0.8,-1.2) -- (4.2,-1.2) -- cycle;
\node at (-2.5,-2) {$\Delta_{+} \subset {\color{red}U_{+}}$};
\node at (-2.5,2.3) {${\hDelta_{+}} \subset {\color{blue}\hV_{+}}$};
\node at (2.5,-2) {$\Delta_{-} \subset {\color{blue}U_{-}}$};
\node at (2.5,2.3) {${\hDelta_{-}} \subset {\color{red}\hV_{-}}$};
\node at (0,3.5) {$[\eta_{g}]$};
\node at (0,-2) {$\PCurr(A)$};
\foreach \a in {a1,a2,b1,b2,c}
	\filldraw (\a) circle [radius=0.075cm];
\end{tikzpicture}
\caption{The set-up of neighborhoods in Theorem~\ref{th:gns}.  For $n \geq M$, the element $\varphi^{n}$ sends the complement of $\hV_{-}$ into $U_{+}$; the element $\varphi^{-n}$ sends the complement of $\hV_{+}$ into $U_{-}$.}\label{fig:gns set-up}
\end{figure}

\begin{proof} 

We replace $\varphi$ by a power so that the results from Section~\ref{subsec:goodness grows} apply.  This is addressed at the end of the proof.

By Lemmas~\ref{lem:nbhds} and \ref{lem:hnbhds} we can assume that $\varphi(U_{+}) \subseteq U_{+}$ and $\hV_{-} \subseteq \varphi(\hV_{-})$.
Let $M$ be the exponent given by Proposition~\ref{prop:backandforth} by using $U_{+} = U_{+}$ and $U_{-} = W = \hV_{-}$. 

For any current
\begin{equation*}
[\mu]\in \varphi^{M}(\PCurr(\FN) - \hV_{-}) = \PCurr(\FN)-\varphi^{M}(\hV_{-}) \subseteq \PCurr(\FN) - W
\end{equation*}
we have $\varphi^{M}[\mu]\in U_{+}$ by Proposition~\ref{prop:backandforth}, as $\varphi^{-M}[\mu]\notin \hV_{-}$. Therefore for any current $[\mu]\in\PCurr(\FN)-\hV_{-}$, we have $\varphi^{2M}[\mu] \in U_{+}$ and hence $\varphi^{2n}[\mu] \in U_{+}$ for all $n \geq M$ as $\varphi(U_{+}) \subseteq U_{+}$.   Therefore, 
\begin{equation*}
\varphi^{2n} (\PCurr(\FN) - \hV_{-}) \subset U_{+}
\end{equation*}
for all $n\ge M$.  A symmetric argument for $\varphi^{-1}$ shows that $\varphi^{2}$ acts with generalized north-south dynamics. We then invoke \cite[Proposition 3.4]{LU2} to deduce that $\varphi$ (and also the original outer automorphism as well) acts with generalized north-south dynamics. 
\end{proof} 

We conclude this section with the analog to Lemma~\ref{lem:dynamics in simplex} regarding the behavior of length under iteration of $\varphi$ that is needed for Theorem~\ref{prop:atoroidal}.  In this statement and its proof, we assume $\varphi \in \Out(\FN)$ satisfies the hypotheses of Theorem~\ref{th:gns} and $\Delta_{\pm}$, $\hDelta_{\pm}$ are the $\varphi$--invariant simplices in $\PCurr(\FN)$ appearing in the statement of that theorem. 

\begin{lemma}\label{lem:growth outside of nbhd}
For each $C > 0$ and neighborhood $\hV \subset \PCurr(\FN)$ of $\hDelta_{-}$ there is a constant $M > 0$ such that if $[\mu] \notin \hV$, then $\wght{\varphi^{n}\mu} \geq C\wght{\mu}$ for all $n \geq M$.
\end{lemma}

\begin{proof}
There is a constant $P$ such that for each current $[\nu] \in \Delta_{+}^{(0)}$ there is a real number $\lambda_{\nu} > 1$ such that $\varphi^{P}\nu = \lambda_{\nu}\nu$~\cite[Remark~6.5]{LU2}.  Let $\lambda_{0} = \min\{\lambda_{\nu} \mid [\nu] \in \Delta_{+}^{(0)} \}$ and $B_{0}$ be large enough so that $\lambda_{0}^{B_{0}} \geq 3$.  Hence $\wght{\varphi^{PB_{0}}\nu} \geq 3\wght{\nu}$ for any $[\nu] \in \Delta_{+}^{(0)}$.  Since the weight function is linear, for any $[\mu] \in \Delta_{+}$ we have $\wght{\varphi^{PB_{0}}\mu} \geq 3\wght{\mu}$ too.
    
Hence there is a neighborhood $U \subseteq \PCurr(\FN)$ of $\Delta_{+}$ such that $\wght{\varphi^{PB_{0}}\mu} \geq 2\wght{\mu}$ for all $[\mu] \in U$. 
By replacing $U$ with a smaller neighborhood, we may assume $\varphi(U) \subseteq U$ and $U \cap \Delta_{-} = \emptyset$ by Lemma~\ref{lem:nbhds}.  Hence $\wght{\varphi^{aPB_{0}}\mu} \geq 2^{a}\wght{\mu}$ for $[\mu] \in U$. 
Let $K = \inf\{ \wght{\varphi^{i}\mu}/\wght{\mu} \mid [\mu] \in U, \, 0 \leq i < PB_{0} \}$.
   
Let $M_{0}$ be the constant from Theorem~\ref{th:gns} applied to the neighborhoods $U$ and $\hV$.  As $\PCurr(\FN)$ is compact, there is a constant $L > 0$ such that $\wght{\varphi^{M_{0}}\mu} \geq L \wght{\mu}$ for all $[\mu] \in \PCurr(\FN)$.     

Let $B_{1}$ be large enough so that $2^{B_{1}}KL \geq C$ and set $M = PB_{0}B_{1} + M_{0}$.  If $n \geq M$, we can write $n = aPB_{0} + i + M_{0}$ where $a \geq B_{1}$ and $0 \leq i < PB_{0}$.  Then for $[\mu] \notin \hV$, we have $[\varphi^{M_{0}}\mu], [\varphi^{i + M_{0}}\mu] \in U$ and so
\begin{equation*}
\wght{\varphi^{n}\mu} \geq 2^{a}\wght{\varphi^{i + M_{0}}\mu} \geq 2^{a}K\wght{\varphi^{M_{0}}\mu} \geq 2^{a}KL\wght{\mu} \geq C\wght{\mu}.\qedhere
\end{equation*}
\end{proof}


\section{Pushing past single-edge extensions}\label{sec:push past one edge}

In this section we apply Theorem~\ref{th:gns} to deal with the case of pushing past single-edge extensions.  Here we use the action on the space of currents to demonstrate that an element is atoroidal.  Given a single-edge extension $\calF_{0} \sqsubset \calF_{1}$ invariant under $\calH$ and $\varphi \in \calH$ such that $\varphi\big|_{\calF_{0}}$ is atoroidal, if there is some nontrivial $g \in \FN$ whose conjugacy class is $\varphi$--periodic, we will either find a finite index subgroup of $\calH$ that fixes $[g]$, or an element $\psi \in \calH$ so that we can play ping-pong with $\varphi$, $\psi \varphi \psi\inv$ to produce an element which is atoroidal on $\calF_{1}$.

To begin, we need a lemma that sets up the appropriate conditions for playing ping-pong.

\begin{lemma}\label{lem:conjugating element}
Suppose $\calF_{0} \sqsubset \calF_{1}$ is a handle extension that is invariant under $\calH < \IA_{N}(\ZZ/3)$ and $\varphi \in \calH$ is such that $\varphi\big|_{\calF_{0}}$ is atoroidal.  Assume $\varphi\big|_{\calF_{1}}$ is not atoroidal and let $[A] \in \calF_{0}$ and $g \in \FN$ be as given by Proposition~\ref{prop:co rank one atoroidal}\eqref{item:one edge fixed} and denote $F = A \ast \I{g}$.  Let $\Delta_{+}(A)$ and $\Delta_{-}(A)$ be the inclusion to $\PCurr(F)$ of the invariant simplices in $\PCurr(A)$ from Theorem~\ref{dynamicsofhyp} for $\varphi\big|_{A}$ and for each other $[B] \in \calF_{0}$, let $\Delta_{+}(B)$ and $\Delta_{-}(B)$ be the invariant simplices in $\PCurr(B)$ from Theorem~\ref{dynamicsofhyp} for $\varphi\big|_{B}$.   Either:
\begin{enumerate}
\item there is a finite index subgroup $\calH'$ of $\calH$ such that $\calH'[g] = [g]$; or
\item there is a $\psi \in \calH$ such that $\psi[g] \neq [g]$ and $\Delta_{+}(B) \cap \psi\big|_{B} \Delta_{-}(B) = \Delta_{-}(B) \cap \psi\big|_{B}\Delta_{+}(B) = \emptyset$ for all $[B] \in \calF_{0}$ \textup{(}including $[A]$\textup{)}. 
\end{enumerate}
\end{lemma}

\begin{proof}
Consider the orbit of the conjugacy class $[g]$ under $\calH$.  If the orbit is finite, then there is a finite index subgroup $\calH'$ of $\calH$ that fixes $[g]$ and so (1) holds.

Else, there is an infinite set $X \subseteq \calH$ such that $h_{1}[g] \neq h_{2}[g]$ for all distinct $h_{1}, h_{2} \in X$.  We claim that there is a pair $h_{1}, h_{2} \in X$ such that $\psi = h_{2}\inv h_{1}$ satisfies the conclusion (2).  By construction of $X$, we have $h_{2}^{-1}h_{1}[g] \neq [g]$ for all distinct $h_{1},h_{2} \in X$ and so we only need to concern ourselves with the intersection of the simplices.  To ease notation here, we will implicitly be using the appropriate restrictions of the elements in $X$. 

To this end, we first consider the vertices $\Delta_{\pm}(B)^{(0)}$ for each $[B] \in \calF_{1}$, i.e., the extremal measures in $\Delta_{\pm}(B)$.  For each such extremal measure $[\mu]$, the support $\supp([\mu])$ contains a sublamination that is uniquely ergodic. Indeed, any such measure comes from an aperiodic $EG$ stratum $H_r$ in the $\CT$ that represents $\varphi$ \cite[Remark 3.4 and Definition 3.5]{U3}. The restriction of $\varphi$ to each $\varphi$--invariant minimal free factor $B_0$ contained in $\pi_1(G_r)$ is both fully irreducible and atoroidal. The support $\supp(\mu_0)$ of the corresponding attracting current $[\mu_0]$ is contained in the support of $[\mu]$, and $\supp(\mu)$ is uniquely ergodic \cite[Proposition 4.4]{Uyaiwip}. 

The fact that $\supp(\mu_0)\subset \supp(\mu)$ follows from the following facts. Recall that for any $\nu\in\Curr(\FN)$,  
$\supp(\nu)$ consists of all bi-infinite paths $\beta$ such that for any finite subpath $\gamma$ of $\beta$ $\langle \gamma, \nu\rangle >0$ \cite[Lemma 3.7]{KL3}. Note that by definition the bi-infinite path $\beta$ obtained by iterating an edge $e$ in an $\EG$ stratum is in the support of the corresponding current. Further, for $e\in H_r$, the attracting lamination corresponding to $H_r$ is the closure of $\beta$ \cite[Lemma 3.1.10 and  Lemma 3.1.15]{BFH00}. The attracting lamination corresponding to a minimal stratum on which $H_r$ maps over is precisely the support of $\mu_0$, hence 
\[
\supp(\mu_0)=\Lambda(B_0, \varphi)\subset\Lambda(\pi_{1}(G_r), \varphi).
\]

Moreover, there are only finitely many such sublaminations.  We set $E_{\varphi}$ to be the set of projective classes of currents obtained by restricting an extremal measure in some $\Delta_{\pm}(B)^{(0)}$ to a uniquely ergodic sublamination contained in its support.

Since the set $E_{\varphi}$ is finite, we can replace $X$ with an infinite subset (which we will still denote $X$) such that for each $s \in E_{\varphi}$ either $h_{1}s = h_{2}s$ for all $h_{1},h_{2} \in X$ or $h_{1}s \neq h_{2}s$ for all distinct $h_{1},  h_{2} \in X$.  Let $E_{1} \subseteq E_{\varphi}$ be the subset for which the first alternative occurs and $E_{\infty} = E_{\varphi} - E_{1}$.

Next fix an arbitrary $h_{1} \in X$ and for each $s \in E_{\infty}$ let \begin{equation*}
X_{s} = \{ h \in X \mid h_{1}s = hs' \text{ for some } s' \in E_{\infty} \}.
\end{equation*}
Notice that each $X_{s}$ is finite set.  Take $h_{2} \in X - \bigcup_{s \in E_{\infty}} X_{s}$.  Then for any $s \in E_{\infty}$ we have $h_{1}s \neq h_{2}s'$ for any $s' \in E_{\infty}$.  If $h_{1}s = h_{2}s'$ for some $s' \in E_{1}$, then $s = h_{1}\inv h_{2}s' = s'$, contradicting the fact that $s \in E_{\infty}$.  Therefore $h_{2}\inv h_{1} s \notin E_{\varphi}$ for all $s \in E_{\infty}$ and $h_{2}\inv h_{1}s = s$ for all $s \in E_{1}$. 

Set $\psi = h_{2}^{-1}h_{1}$.  We have that for any $s \in E_{\varphi}$, either $\psi s = s$ or $\psi s \notin E_{\varphi}$.  

Now take $[\mu] \in \Delta_{-}(B)$ for some $[B] \in \calF_{1}$ and suppose that $\psi[\mu] \in \Delta_{+}(B)$.  Therefore we can write $\mu = \sum_{i=1}^{m} a_{i} \mu_{i}^{-}$ for some extremal measures $[\mu_{i}] \in \Delta_{-}(B)^{(0)}$ and coefficients $a_{i} > 0$.    Hence we have:
\begin{equation*}
\sum_{i=1}^{m} a_{i}\psi\mu_{i}^{-} = \psi\mu = \sum_{j=1}^{n} b_{j}\mu_{j}^{+}
\end{equation*}
for some extremal measures $[\mu_{j}^{+}] \in \Delta_{+}(B)^{(0)}$ and coefficients $b_{j} > 0$.  In particular the union of the supports of $\supp(\psi \mu_{i}^{-})$ for $i = 1,\ldots,m$ equals the union of the supports $\supp(\mu_{j}^{+})$ for $j = 1,\ldots,n$.  Let $\Lambda \subseteq \supp(\mu_{1}^{-})$ be a uniquely ergodic sublamination.  As uniquely ergodic laminations are minimal, $\psi \Lambda$ is a sublamination of $\supp(\mu_{j}^{+})$ for some $j$.  Thus $\psi[\mu_{1}^{-}\big|_{\Lambda}] = [\mu_{j}^{+}\big|_{\Lambda}]$.  This is a contradiction as $[\mu_{1}^{-}\big|_{\Lambda}], [\mu_{j}^{+}\big|_{\Lambda}] \in E_{\varphi}$ are distinct. 
\end{proof}

We can now play ping-pong to construct atoroidal elements.

\begin{proposition}\label{prop:atoroidal}
Suppose $\calF_{0} \sqsubset \calF_{1}$ is a single-edge extension that is invariant under $\calH < \IA_{N}(\ZZ/3)$ and $\varphi \in \calH$ is such that $\varphi\big|_{\calF_{0}}$ is atoroidal.  Assume $\varphi\big|_{\calF_{1}}$ is not atoroidal and let $[g]$ be the fixed conjugacy class in $\FN$ given by Proposition~\ref{prop:co rank one atoroidal}\eqref{item:one edge fixed}.  Either:
\begin{enumerate}
\item there is a finite index subgroup $\calH'$ of $\calH$ such that $\calH'[g] = [g]$; or
\item there is a $\psi \in \calH$ and a constant $M > 0$ such that $(\theta^{m}\varphi^{n})\big|_{\calF_{1}}$ is atoroidal for any $m,n \geq M$  where $\theta = \psi\varphi\psi^{-1}$.
\end{enumerate}
\end{proposition}

\begin{proof}
Assume (1) does not hold.  Let $\psi \in \calH$ be the element given by Lemma~\ref{lem:conjugating element} and set $\theta = \psi\varphi\psi^{-1}$.  Also, let $[A] \in \calF_{0}$ be the free factor given by Proposition~\ref{prop:co rank one atoroidal} and denote $F = A \ast \I{g}$.  Notice that $\theta\big|_{B}$ is atoroidal for all $[B] \in \calF_{0}$ and $[g'] = \psi[g] \neq [g]$ is the only conjugacy class in $\calF_{1}$ fixed by $\theta$ up to taking powers and inversion.    We will show that for sufficiently large $m$ and $n$ and any $[B] \in \calF_{1}$ the element $(\theta^{m}\varphi^{n})\big|_{B}$ does not have any non-zero fixed points in $\Curr(B)$.

For each $[B] \in \calF_{0}$, let $\Delta_{\pm}(B)$ be the invariant simplices as defined in Lemma~\ref{lem:conjugating element}.  By this lemma we have that $\Delta_{+}(B) \cap \psi\big|_{B}\Delta_{-}(B) = \Delta_{-}(B) \cap \psi\big|_{B}\Delta_{+}(B) = \emptyset$ for any $[B] \in \calF_{0}$.  To begin, we will assume that $\calF_{0} = \{[A]\}$, $\calF_{1} = \{[F]\}$ and to simplify notation, we will implicitly use the restrictions of the elements to $F$.  

There are open sets $U, V, \hU, \hV \subset \PCurr(F)$ such that:
\begin{enumerate}
\item $\Delta_{+} \subset U$, $\hDelta_{+} \subset \hU$, $\Delta_{-} \subset V$ and $\hDelta_{-} \subset \hV$;
\item $U \subseteq \hU$, $V \subseteq \hV$; and
\item $\hU \cap \psi\hV = \emptyset$ and $\psi \hU \cap \hV = \emptyset$.
\end{enumerate}
See Figure~\ref{fig:nbhd set-up}.  
\begin{figure}[h!]	
\centering
\begin{tikzpicture}[every node/.style={inner sep=0pt},scale=0.9]
\draw[very thick] (-5,4) rectangle (5,-4);
\filldraw[black!20!white] (-3.25,1.75) -- (0,0.875) -- (-3.25,0) -- cycle;
\filldraw[black!20!white] (-3.25,-1.75) -- (0,-0.875) -- (-3.25,0) -- cycle;
\filldraw[black!20!white] (3.25,1.75) -- (0,0.875) -- (3.25,0) -- cycle;
\filldraw[black!20!white] (3.25,-1.75) -- (0,-0.875) -- (3.25,0) -- cycle;
\draw[very thick] (-3.25,1.75) -- (-3.25,-1.75) node[pos=0] (a1) {} node[pos=0.5] (b1) {} node[pos=1] (c1) {};
\draw[very thick] (3.25,1.75) -- (3.25,-1.75) node[pos=0] (a2) {} node[pos=0.5] (b2) {} node[pos=1] (c2) {};
\node at (0,-0.875) (d2) {};
\draw[very thick] (a1) -- (b2) node[pos=0.5] (d1) {};
\draw[very thick] (b1) -- (a2);
\draw[very thick] (b1) -- (c2) node[pos=0.5] (d2) {};
\draw[very thick] (c1) -- (b2);
\draw[thick,dashed,blue,rounded corners=20pt] (-3.8,2.45) -- (0.85,0.875) -- (-3.8,-0.7) -- cycle;
\draw[thick,dashed,blue,rounded corners=20pt] (-3.8,-2.45) -- (0.85,-0.875) -- (-3.8,0.7) -- cycle;
\draw[thick,dashed,red,rounded corners=20pt] (3.8,2.45) -- (-0.85,0.875) -- (3.8,-0.7) -- cycle;
\draw[thick,dashed,red,rounded corners=20pt] (3.8,-2.45) -- (-0.85,-0.875) -- (3.8,0.7) -- cycle;
\draw[thick,dashed,red,rounded corners=5pt] (-3.5,2) -- (-3,2) --(-3,-0.25) -- (-3.5,-0.25) -- cycle;
\draw[thick,dashed,red,rounded corners=5pt] (-3.5,-2) -- (-3,-2) --(-3,0.25) -- (-3.5,0.25) -- cycle;
\draw[thick,dashed,blue,rounded corners=5pt] (3.5,2) -- (3,2) --(3,-0.25) -- (3.5,-0.25) -- cycle;
\draw[thick,dashed,blue,rounded corners=5pt] (3.5,-2) -- (3,-2) --(3,0.25) -- (3.5,0.25) -- cycle;
\node at (-4.35,0.875) {$\Delta_{+}$};
\node at (-4.35,-0.875) {$\psi\Delta_{+}$};
\node at (4.35,0.875) {$\Delta_{-}$};
\node at (4.35,-0.875) {$\psi\Delta_{-}$};
\node at (-2,2.5) {${\color{red}U} \subset {\color{blue}\hU}$};
\node at (-2,-2.5) {$\psi{\color{red}U} \subset \psi{\color{blue}\hU}$};
\node at (2,2.5) {${\color{blue}V} \subset {\color{red}\hV}$};
\node at (2,-2.5) {$\psi{\color{blue}V} \subset \psi{\color{red}\hV}$};
\node at (0,1.8) {$[\eta_{g}]$};
\node at (0,-1.8) {$\psi[\eta_{g}]$};
\foreach \a in {a1,a2,b1,b2,c1,c2,d1,d2}
	\filldraw (\a) circle [radius=0.075cm];
\end{tikzpicture}
\caption{The set-up of neighborhoods in $\PCurr(F)$ for Proposition~\ref{prop:atoroidal}.}\label{fig:nbhd set-up}
\end{figure}

Let $M_{0}$ be the constant from Theorem~\ref{th:gns} applied to $\varphi$ with $U$ and $\hV$.  Let $M_{1}(\varphi)$, $M_{1}(\theta)$ respectively, be the constants from Lemma~\ref{lem:growth outside of nbhd} applied to $\varphi$ with $\hV$, $\theta$ with $\psi\hV$ respectively with $C = 2$.  Likewise, let $M_{1}(\varphi\inv)$, $M_{1}(\theta\inv)$ respectively, be the constants from Lemma~\ref{lem:growth outside of nbhd} applied to $\varphi\inv$ and $\hU$, $\theta\inv$ and $\psi\hU$ respectively with $C = 2$.

Set $M = \max\{M_{0},M_{1}(\varphi),M_{1}(\theta),M_{1}(\varphi\inv),M_{1}(\theta\inv)\}$ and suppose $m,n \geq M$.  Let $\mu \in \Curr(F)$ be non-zero.

If $[\mu] \notin \hV$, then $\varphi^{n}[\mu] \in U$ (Theorem~\ref{th:gns}) and $\wght{\varphi^{n}\mu} \geq 2\wght{\mu}$ (Lemma~\ref{lem:growth outside of nbhd}).  Further $\varphi^{n}[\mu] \notin \psi\hV$ and so $\wght{\theta^{m}\varphi^{n}\mu} \geq 2\wght{\varphi^{n}\mu} \geq 4\wght{\mu}$ (Lemma~\ref{lem:growth outside of nbhd} again).  Hence $\theta^{m}\varphi^{n}\mu \neq \mu$.

Else $[\mu] \in \hV$ and so $[\mu] \notin \psi\hU$.  Hence $\theta^{-m}[\mu] \in \psi V$ (Theorem~\ref{th:gns}) and $\wght{\theta^{-m}\mu} \geq 2\wght{\mu}$ (Lemma~\ref{lem:growth outside of nbhd}).  Further $\theta^{-m}[\mu] \notin \hU$ and so $\wght{\varphi^{-n}\theta^{-m}\mu} \geq 2\wght{\theta^{-m}\mu} \geq 4\wght{\mu}$ (Lemma~\ref{lem:growth outside of nbhd} again).  Hence $\theta^{m}\varphi^{n}\mu \neq \mu$.  

Therefore $(\theta^{m}\varphi^{n})\big|_{F}$ is atoroidal.

The general case is a straight forward modification, additionally playing ping-pong simultaneously in each $\Curr(B)$ for $[B] \in \calF_{0} - \{[A]\}$ using Theorem~\ref{dynamicsofhyp} in place of Theorem~\ref{th:gns} and Lemma~\ref{lem:dynamics in simplex} in place of Lemma~\ref{lem:growth outside of nbhd}. 
\end{proof}

Putting together the previous results, we get the following proposition which allows us to push past single-edge extensions.  Care needs to be taken to avoid distributing the action on other extensions which adds a layer of technicality.

\begin{proposition}\label{prop:inductive step}
Suppose $\calH < \IA_{N}(\ZZ/3)$.  Let \[\emptyset = \calF_{0} \sqsubset \calF_{1} \sqsubset \cdots \sqsubset \calF_{k} = \{[ \FN] \}\] be an  $\calH$--invariant filtration by free factor systems and suppose $\calF_{i-1} \sqsubset \calF_{i}$ is a single-edge extension for some $2\le i\le k$.  Suppose there exists some $\varphi \in \calH$ such that:
\begin{enumerate}
\item[(a)] the restriction of $\varphi$ to $\calF_{i-1}$ is atoroidal; and
\item[(b)] $\varphi$ is irreducible and non-geometric with respect to each multi-edge extension $\calF_{j-1} \sqsubset \calF_{j}$, $j = 1,\ldots,k$.
\end{enumerate}
Then either:
\begin{enumerate}
\item\label{first alternative} there is a finite index subgroup $\calH'$ of $\calH$ and a nontrivial element $g \in \FN$ such that $\calH'[g] = [g]$; or
\item\label{second alternative} there exists an element $\hat\varphi \in \calH$ such that:
\begin{enumerate}
\item[i.] the restriction of $\hat\varphi$ to $\calF_{i}$ is atoroidal; and
\item[ii.] $\hat\varphi$ is irreducible and non-geometric with respect to each multi-edge extension $\calF_{j-1} \sqsubset \calF_{j}$, $j = 1,\ldots,k$.
\end{enumerate}
\end{enumerate}
\end{proposition}

\begin{proof}
As mentioned in Section~\ref{subsec:free factor}, there are three types of single-edge extensions.  We deal with these separately.

If $\calF_{i-1} \sqsubset \calF_{i}$ is a circle extension, then $\calF_{i} = \calF_{i-1} \cup \{[\I{g}]\}$ for some nontrivial element $g \in \FN$.  As both $\calF_{i-1}$ and $\calF_{i}$ are $\calH$--invariant, we have $\calH[g] = [g]$ and so~\eqref{first alternative} holds.

If $\calF_{i-1} \sqsubset \calF_{i}$ is a barbell extension then by Proposition~\ref{prop:co rank one atoroidal}, $\varphi\big|_{\calF_{i}}$ is atoroidal.  Hence we may take $\hat\varphi = \varphi$ to satisfy~\eqref{second alternative}.

Lastly, we assume that $\calF_{i-1} \sqsubset \calF_{i}$ is a handle extension.  If $\varphi\big|_{\calF_{i}}$ is atoroidal, then $\hat\varphi = \varphi$ satisfies~\eqref{second alternative}.  Else, by Proposition~\ref{prop:atoroidal}, either there is a finite index subgroup $\calH'$ of $\calH$ such that $\calH'[g] = [g]$ or there is an element $\psi \in \calH$ and constant $M$ such that $(\theta^{m}\varphi^{n})\big|_{\calF_{1}}$ is atoroidal for $m,n \geq M$ where $\theta = \psi \varphi \psi\inv$.  

If the finite index subgroup $\calH'$ exists, then clearly~\eqref{first alternative} holds and hence, we assume the existence of the element $\psi \in \calH$ and constant $M$ with the properties above.  Let $S = \{ j \mid \calF_{j-1} \sqsubset \calF_{j} \text{ is multi-edge} \}$.  What remains to show is that for some $m,n \geq M$ the element $\theta^{m}\varphi^{n}$ is irreducible and non-geometric with respect to $\calF_{j-1} \sqsubset \calF_{j}$ for all $j \in S$.  

Suppose $j \in S$.  As in~\cite[Theorem~6.6]{CU}, there is a single component $[B_{j}] \in \calF_{j}$ that is not a component of $\calF_{j-1}$ and subgroups $A_{j,1},\ldots,A_{j,k} < B_{j}$ where $\{[A_{j,1}] ,\ldots, [A_{j,k}]\} \subseteq \calF_{j-1}$ such that for $\calA_{j}$, the free factor system in $B_{j}$ determined by $A_{j,1},\ldots, A_{j,k}$, the restriction $\varphi\big|_{B_{j}} \in \Out(B_{j};\calA_{j})$ is irreducible and non-geometric.  Let $X_{j} = \zF(B_{j};\calA_{j})$ be the $\delta$--hyperbolic graph given by Theorem~\ref{th:relativecosurface}.  Notice that by (b), the element $\varphi$ and its conjugate $\theta$ act as hyperbolic isometries on $X_{j}$.  The remainder of the argument is an easy exercise using $\delta$--hyperbolic geometry, we sketch the details.  

Recall that two hyperbolic isometries of a $\delta$--hyperbolic space $X$ are said to be \emph{independent} if their fixed point sets in $\partial X$ are disjoint and \emph{dependent} otherwise.  Let $I \subseteq S$ be the subset of indices where $\varphi$ and $\theta$ are independent and $D = S - I$.  By \cite[Proposition~4.2]{CU} and \cite[Theorem~3.1]{CU}, there are constants $m,n_{0} \geq M$ such that $\theta^{m}\varphi^{n}$ acts hyperbolically on $X_{j}$ if $j \in I$ and $n \geq n_{0}$.  Then, by \cite[Proposition~3.4]{CU}, there is an $n \geq n_{0}$ such that $\theta^{m}\varphi^{n}$ acts hyperbolically on $X_{j}$ if $j \in D$.  By Theorem~\ref{th:relativecosurface}, the element $\theta^{m}\varphi^{n}$ is irreducible and non-geometric with respect to each $\calF_{j-1} \sqsubset \calF_{j}$ when $j \in S$.  This shows that~\eqref{second alternative} holds.   
\end{proof}


\section{Proof of the subgroup alternative}\label{sec:proof}

In this section, we complete the proof of the main result of this article.

\begin{restate}{Theorem}{th:alternative} 
Let $\calH$ be a subgroup of $\Out(\FN)$ where $N \geq 3$.  Either $\calH$ contains an atoroidal element or there exists a finite index subgroup $\calH'$ of $\calH$ and a nontrivial element $g \in \FN$ such that $\calH'[g] = [g]$.
\end{restate}

\begin{proof}
Without loss of generality, we may assume that $\calH < \IA_{N}(\ZZ/3)$.  Let $\emptyset = \calF_{0} \sqsubset \calF_{1} \sqsubset \cdots \sqsubset \calF_{m} = \{[ \FN] \}$ be a maximal $\calH$--invariant filtration by free factor systems.  By the Handel--Mosher Subgroup Decomposition, for each $\calF_{i-1} \sqsubset \calF_{i}$ which is a multi-edge extension, $\calH$ contains an element which is irreducible with respect to this extension~\cite[Theorem~D]{HMIntro}.  

Suppose that there is no finite index subgroup $\calH'$ of $\calH$ and nontrivial $g \in \FN$ such that $\calH'[g] = [g]$.   In particular, every multi-edge extension $\calF_{i-1} \sqsubset \calF_{i}$ is non-geometric by Theorem~\ref{th:multi edge dichotomy}.  Therefore, by Corollary~\ref{co:HM-simultaneous} there is a $\varphi \in H$ that is irreducible and non-geometric with respect to each multi-edge extension $\calF_{j-1} \sqsubset \calF_{j}$ for $j = 1,\ldots,m$.

We claim that for each $i = 1,\ldots, m$ there is an $\varphi_{i} \in \calH$ whose restriction to $\calF_{i}$ is atoroidal and is irreducible and non-geometric with respect to each multi-edge extension $\calF_{j-1} \sqsubset \calF_{j}$ for $j = 1,\ldots,m$.

Indeed, by our assumptions, $\emptyset = \calF_{0} \sqsubset \calF_{1}$ must be a multi-edge extension and so we can take $\varphi_{1} = \varphi$.

Now assume that $\varphi_{i-1}$ exists.  If $\calF_{i-1} \sqsubset \calF_{i}$ is a single-edge extension, we apply Proposition~\ref{prop:inductive step} to $\varphi = \varphi_{i-1}$ and set $\varphi_{i} = \hat\varphi$.  Else, $\calF_{i-1} \sqsubset \calF_{i}$ is a multi-edge extension and we apply Lemma~\ref{co:non-geometric atoroidal} to $\varphi_{i-1}$ and the extension $\calF_{i-1} \sqsubset \calF_{i}$ to conclude that we may set $\varphi_{i} = \varphi_{i-1}$ in this case.  

Thus the elements $\varphi_{i}$ as claimed exist.  By construction, the element $\varphi_{m} \in \calH$ is atoroidal. 
\end{proof}


\bibliography{bib}
\bibliographystyle{acm}

\end{document}